\documentclass[12pt]{article}

\usepackage{amsmath,a4wide}
\usepackage{graphicx,amsmath}
\usepackage{amssymb}
\usepackage{amsthm}
\usepackage[margin=1in]{geometry}  
\usepackage{amsfonts}              

\newtheorem{defi}{Definition}[section]
\newtheorem{thm}{Theorem}[section]

\newtheorem{lem}{Lemma}[section]

\newtheorem{prop}{Proposition}[section]

\newcommand{\RR}{\mathbb{R}}      
\newcommand{\NN}{\mathbb{N}}
\newcommand{\PP}{\mathcal{P}}     
\newcommand{\RA}{\rightarrow}     
\newcommand{\td}{\nabla}   
\newcommand{\vp}{\varphi}

\makeatletter
\numberwithin{equation}{section}
\makeatother

\begin{document}

\nocite{*}

\title{Radial weak solutions for the Perona-Malik equation as a differential inclusion}

\author{Seonghak Kim and Baisheng Yan\\[1ex] Department of Mathematics\\
Michigan State University\\
East Lansing, Michigan 48824 USA}

\maketitle

\begin{abstract}
  The Perona-Malik equation is an ill-posed forward-backward parabolic equation with major application in image processing. In this paper we study the Perona-Malik type equation and show that, in all dimensions,  there exist infinitely many radial weak solutions to  the homogeneous Neumann boundary problem for any smooth nonconstant radially symmetric initial data. Our approach is to reformulate the $n$-dimensional equation into a one-dimensional equation, to convert the one-dimensional problem into a differential inclusion problem, and to apply a Baire's category method to generate infinitely many solutions.
\end{abstract}

\tableofcontents


\section{Introduction}

In this paper we investigate the existence of weak solutions for an $n$-dimensional Perona-Malik type equation under the homogeneous Neumann boundary condition and radially symmetric initial data:
\begin{equation}\label{1-1}
\left\{ \begin{array}{ll}
  u_t=\mathrm{div}(a(|Du|^2)Du) & \textrm{in }\Omega_T:=\Omega\times(0,T) \\
  \frac{\partial u}{\partial \mathbf{n}}=0 & \textrm{on }\partial\Omega\times(0,T) \\
  u(x,0)=u_0(x) & \textrm{for } x\in\Omega,
\end{array} \right.
\end{equation}
where $\Omega:=B_R(0)$ is the open ball in $\RR^n$ ($n\ge 1$) with center $0$ and radius $R>0$,  $T>0$ is a given time,  $\mathbf{n}$ is outward unit normal to $\partial\Omega$, $u_0\colon \Omega\RA\RR$ is a radially symmetric initial function, and $a\in C^{2,\alpha}([0,\infty))$, for some $\alpha\in (0,1)$, is a positive function satisfying the following:
\begin{equation}\label{1-2}
  2p\,a'(p)+a(p)\left\{ \begin{array}{ll}
  >0 & \textrm{for } 0\leq p<1 \\
  =0 & \textrm{for } p=1 \\
  <0 & \textrm{for } p>1,
\end{array} \right.  \quad \mbox{and} \quad
  \lim_{p\RA\infty}\sigma(p)=0,
\end{equation}
where $\sigma(p)=a(p^2)p$ for $p\in \RR.$  We can relax the function $a$ in (\ref{1-2}) to $a\in C^{2,\alpha}([0,1))\cap C([0,\infty))$ with $\sigma$ strictly decreasing on $[1,\infty)$ without affecting the result of this paper. The notation and assumptions in this paragraph will be kept throughout the paper unless otherwise stated.
\begin{figure}[ht]
\begin{center}
\includegraphics[scale=0.6]{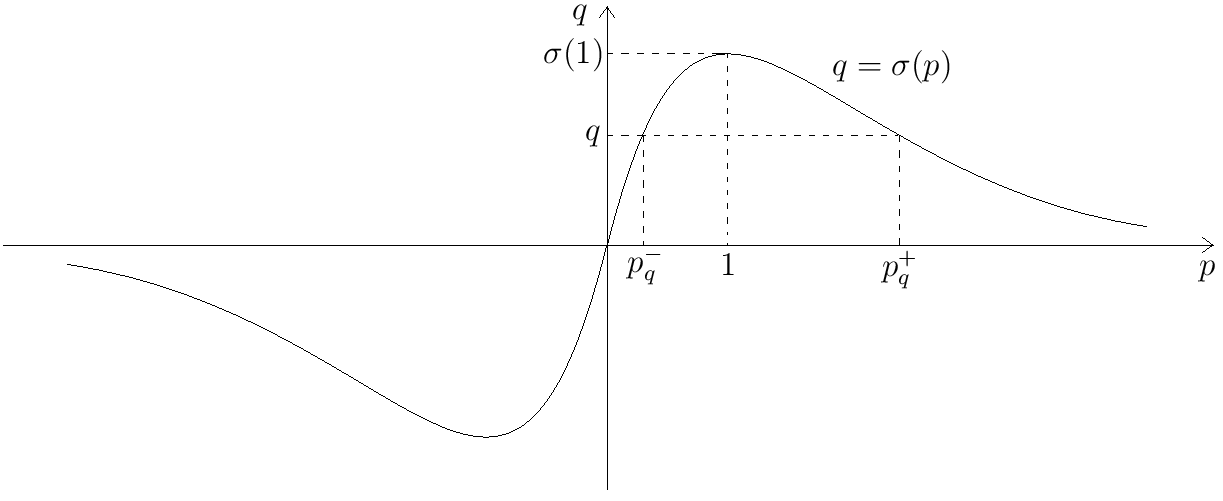}
\end{center}
\caption{The graph of a typical function $q=\sigma(p)$}
\label{fig1}
\end{figure}

In the original paper of \textsc{Perona \& Malik} \cite{perona90}, they proposed an anisotropic diffusion model (\ref{1-1}) for an edge enhancement of computer vision, where $\Omega\subset\RR^2$ is a square and $a(p)$ is given as
$$
\textrm{either}\,\,\,\,a(p)=\frac{1}{1+\frac{p}{k^2}}\,\,\,\,\textrm{or}\,\,\,\,a(p)=\mathrm{exp}\left(-\frac{p}{2k^2}\right),
$$
with the fixed threshold $k>0$  according to some experimental purposes. In our case we have chosen $k=1$ for simplicity, but this choice is not essential.

For the general discussion, let us assume for the moment that $\Omega\subset\RR^n$ is a bounded $C^1$ domain and that $a(p)=(1+p)^{-1}$. Given a point $x\in\bar{\Omega}$, we say that the initial condition $u_0\in C^1(\bar\Omega)$ is \emph{subcritical} at $x$ if $|Du_0(x)|<1$, \emph{supercritical} at $x$ if $|Du_0(x)|>1$, and \emph{critical} at $x$ if $|Du_0(x)|=1$. The initial condition $u_0$ is \emph{transcritical} in $\Omega$ if there are two points $x,y\in\Omega$ with $|Du_0(x)|<1$ and $|Du_0(y)|>1$.
Existence of   global or local classical solutions to Problem (\ref{1-1})  depends heavily on the initial condition $u_0$. \textsc{Kawohl \& Kutev} \cite{kawohl98} showed that a global classical solution exists in any dimension if $u_0$ is subcritical in $\bar{\Omega}$. They also proved that the problem  cannot admit a global classical solution for $n=1$ if $u_0$ is transcritical in $\Omega$ under some technical assumptions, and these assumptions were completely removed later by \textsc{Gobbino} \cite{gobbino07}. Concerning the Perona-Malik type equation, it  had been the general belief  that  classical solutions can only exist if the initial data are smooth, even analytic, at supercritical points; this was formally streamlined in  \textsc{Kichenassamy} \cite{kichenassamy97}. As regards to the class of suitable initial conditions for classical solutions of (\ref{1-1}), \textsc{Ghisi \& Gobbino} \cite{ghisi09} has recently established that for $n=1$,  the set of initial conditions for which the problem (\ref{1-1}) has a local classical solution is dense in $C^1(\bar\Omega)$.

The situation concerning the existence of a global classical solution to (\ref{1-1}) with a transcritical initial condition for $n\ge 2$ turns out to be quite different from the case $n=1.$ The first existence result of global classical solutions with $u_0$ transcritical for $n\geq2$ was obtained by \textsc{Ghisi \& Gobbino} \cite{ghisi11}, where they constructed  a class of global radial $C^{2,1}$ solutions with suitably chosen radial initial data transcritical on an annulus centered at the origin; these solutions also have the property of finite-time extinction of supercritical region. In contrast  to the one-dimensional result of \cite{gobbino07, kawohl98} mentioned above, their result showed a quite different feature of the higher dimensional problem. On the other hand, in the radial case,  \textsc{Ghisi \& Gobbino}  \cite{ghisi11-1} also proved that  a global $C^1$ solution cannot exist if the gradient of initial condition $u_0$ is very large at a point. Therefore, requirement of  regularity of the solution (e.g., classical or $C^1$) would prevent the existence of such a solution if the initial data should be arbitrarily given and  transcritical.

When the initial condition $u_0$ is any  given smooth function (satisfying certain compatibility condition on $\partial\Omega$), it seems natural to lower the expectation on the regularity of solutions by finding plausible weak solutions to (\ref{1-1}). Even under   the lowering of regularity have  enormous  difficulties  occurred in the existence of weak solutions. Among many different approaches and attempts  in this direction, e.g., the $\Gamma$-limit method in \textsc{Bellettini \& Fusco} \cite{bellettini08}, the Young measure solutions in \textsc{Chen \& Zhang} \cite{chen06}, and numerical scheme analyses in \textsc{Esedoglu} \cite{esedoglu01} and \textsc{Esedoglu \& Greer} \cite{esedoglu09},  to our best knowledge, {\sc Zhang}  \cite{zhang06}  was the first to successfully prove  that,  for $n=1$, there are infinitely many  (Lipschitz) weak solutions to (\ref{1-1}) for any given smooth initial data $u_0$. His method uses the variational technique of  differential inclusion  together with the so called in-approximation method or convex integration; this new method can also deal with other ill-posed forward-backward diffusion problems (see, e.g.,  the pioneering work of \textsc{H\"{o}llig} \cite{hollig83} and its recent generalization by {\sc Zhang} \cite{zhang06-1}.)  In this paper, we primarily  generalize the result of {\sc Zhang} \cite{zhang06} to the case of radial weak solutions to problem (\ref{1-1}) in all dimensions.

\par

For $\alpha\in(0,1)$, we use $C^{3+\alpha,1+\alpha/2}(\bar{\Omega}_T)$ to denote the parabolic H\"{o}lder space of functions $u\in C^0(\bar{\Omega}_T)$ such that $u_t,u_{x_it}, u_{x_i}, u_{x_ix_j}, u_{x_ix_jx_k} \in C^0(\bar{\Omega}_T)$ and such that the quantities
$$
\sup_{x\in\Omega\atop{{s,t\in(0,T),s\neq t}}}\frac{|u_{x_it}(x,s)-u_{x_it}(x,t)|}{|s-t|^{\alpha/2}},\,\,\, \sup_{x\in\Omega\atop{{s,t\in(0,T),s\neq t}}}\frac{|u_{x_ix_jx_k}(x,s)-u_{x_ix_jx_k}(x,t)|}{|s-t|^{\alpha/2}},
$$
$$
\sup_{x,y,\in\Omega,x\neq y\atop{t\in (0,T)}}\frac{|u_{x_it}(x,t)-u_{x_it}(y,t)|}{|x-y|^{\alpha}},\,\,\, \sup_{x,y,\in\Omega,x\neq y\atop{t\in(0,T)}}\frac{|u_{x_ix_jx_k}(x,t)-u_{x_ix_jx_k}(y,t)|}{|x-y|^{\alpha}}
$$
are all finite, where $i,j,k\in\{1,\ldots,n\}$.

We state the main result of this paper as follows.

\begin{thm}\label{1-8}
Let  $u_0\in C^{3,\alpha}(\bar\Omega)$ be a radially symmetric function with \[
M:=\max_{\bar\Omega}|Du_0|>0\]
 such that the compatibility condition holds:
\begin{displaymath}
\frac{\partial u_0}{\partial \mathbf{n}}(x)=0 \,\,\,\,\,\,\forall x\in\partial\Omega.
\end{displaymath}
Then the forward-backward Neumann problem (\ref{1-1}) admits infinitely many radial weak solutions $u\in W^{1,\infty}(\Omega_T)$ satisfying the following:
\begin{itemize}
\item[(a)] For every $\xi\in C^1_0(\Omega_T)$,\\
\begin{eqnarray}\label{1-3}
\int_{\Omega_T}(u_t\xi+a(|Du|^2)Du\cdot D\xi)dxdt=0.
\end{eqnarray}
\item[(b)] The solutions $u$ are (uniformly and locally) classical near $\{\partial\Omega\cup\{0\}\}\times[0,T]$ in the sense that there exists a constant $\delta_0$ with $0<\delta_0<R/2$, independent of $u$, such that
\begin{equation}\label{1-4}
\left\{ \begin{array}{l}
  u\in C^{3+\alpha,1+\alpha/2}\left(\{\overline{B_{\delta_0}(0)}\cup \overline{(B_{R}(0)\setminus B_{R-\delta_0}(0))}\}\times[0,T]\right), \\
  u_t=\mathrm{div}(a(|Du|^2)Du)\,\,\textrm{pointwise in } \{B_{\delta_0}(0)\cup (B_{R}(0)\setminus \overline{B_{R-\delta_0}(0)})\}\times(0,T).
\end{array} \right.
\end{equation}
\item[(c)] The initial condition holds:
\begin{eqnarray}\label{1-6}
u(x,0)=u_0(x)\,\,\,\,\,\,\forall x\in\bar\Omega.
\end{eqnarray}
\item[(d)] The boundary condition is satisfied:
\begin{eqnarray}\label{1-7}
\frac{\partial u}{\partial \mathbf{n}}(x,t)=0\,\,\,\,\,\,\forall (x,t)\in\partial\Omega\times[0,T].
\end{eqnarray}
\item[(e)] The almost maximum principle holds when $u_0$ is critical or supercritical at some point in $\Omega$; that is, if $M\geq 1$, then, given any $\epsilon>0$, we can choose the solutions $u$ to satisfy the following:
\begin{eqnarray}\label{1-5}
||Du||_{L^\infty(\Omega_T;\RR^n)}\leq M+\epsilon.
\end{eqnarray}
\item[(f)] The conservation of mass:
\begin{eqnarray}\label{1-10}
\int_\Omega u(x,t)dx=\int_\Omega u_0(x)dx\,\,\,\,\,\,\forall t\in[0,T].
\end{eqnarray}
\end{itemize}
\end{thm}
The proof of this theorem will be given in Section \ref{s3}.

Observe that when the space dimension $n=1$, Theorem \ref{1-8} implies the main result  of {\sc Zhang} \cite{zhang06} in a bit more general fashion if we cut the space-time domain $\Omega_T=(-R,R)\times(0,T)$ into half, $(0,R)\times(0,T)$; hence our work is indeed a generalization of \cite{zhang06}.

Let us explain our main approach and the major difficulty that arises if $n>1$. One can easily reformulate the  equation in (\ref{1-1}) for radial functions $u(x,t)$  into the one-dimensional equation:
\begin{equation}\label{1-9}
v_t=(a(v_s^2)v_s)_s+a(v_s^2)v_s\frac{n-1}{s}\,\,\,\,\,\,\textrm{in}\,\,(0,R)\times(0,T),
\end{equation}
where $s=|x|$ is the radial variable and $v(s,t)=u(x,t)$. Using the flux function $\sigma(p)=a(p^2)p$ and overlooking the singularity at $s=0$, this equation can be recast as
\[
(s^{n-1}v)_t=(s^{n-1}\sigma(v_s ) )_s\,\,\,\,\,\,\textrm{in}\,\,(0,R)\times(0,T).
\]
Introduce a stream function $\vp$ with $\vp_s=s^{n-1}v$, $\vp_t=s^{n-1}\sigma(v_s)$, and let  $\Phi=(v,\vp).$ Then, to solve the equation (\ref{1-9}) in a weak form,  it is  sufficient to find a function $\Phi=(v,\vp)\in W^{1,\infty}((0,R)\times (0,T);\RR^2)$ with the Jacobian matrix $\nabla  \Phi(s,t)=\left(
                            \begin{array}{cc}
                              v_s & v_t\\
                              \vp_s& \vp_t \\
                            \end{array}
                          \right)$,  such that
\begin{equation}\label{di-0}
\td\Phi(s,t)\in \Sigma(s,v(s,t))\,\,\,\,\,\,\textrm{for a.e.}\,\,(s,t)\in(0,R)\times (0,T),
\end{equation}
where,  for each $s>0$ and each $v\in\RR$, the set $\Sigma(s,v)$ is defined by
\[
\Sigma(s,v):=\left\{\left(
                            \begin{array}{cc}
                              p & l \\
                              s^{n-1}v & s^{n-1}\sigma(p) \\
                            \end{array}
                          \right)\in\RR^{2\times 2}:p,\, l\in\RR
       \right\}.
\]
If $n=1$, the partial differential inclusion (\ref{di-0})  is the same as in \cite{zhang06} since $s^{n-1}=1$, with the set $\Sigma(s,v)$  independent of $s$.  But the presence of the term $s^{n-1}$ for $n\geq 2$ enormously affects the inclusion problem by making it essentially inhomogeneous in the variable $s$. In the fulfillment of the density result, Theorem \ref{2-3-1}, for applying a Baire's category method in Subsection \ref{s2-1}, we have to construct some auxiliary functions as in \cite{zhang06}. Rather substantial difference occurs in the way of defining these functions in Section \ref{s4} as the equation $\vp_s=s^{n-1}v$ should be kept in every gluing process and the term $s^{n-1}$ makes the functions necessarily depend on the position $s$ where they are glued. Accordingly, auxiliary functions are piecewise $C^1$ with proper $s$-derivatives on the regions that are separated by \emph{nonlinear} $C^1$ curves.

The study of inhomogeneous partial differential inclusions of the type (\ref{di-0}) stems from the successful understandings of homogeneous inclusion of the form $Du(x)\in K$ first encountered in the  study of crystal microstructure  by \textsc{Ball \& James}  \cite{ball87,ball92} and \textsc{Chipot \& Kinderlehrer} \cite{chipot88}. Subsequent developments   including some  important  applications and the generalization to  inhomogeneous differential  inclusions of the form $Du(x)\in K(x,u(x))$ have been extensively explored; see, e.g., \textsc{Dacorogna \&  Marcellini} \cite{dacorogna97, dacorogna99},  \textsc{Kirchheim} \cite{kirchheim01}, \textsc{M\"{u}ller \& \v{S}ver\'{a}k} \cite{muller96, muller99, muller03},  \textsc{M\"{u}ller \& Sychev} \cite{muller01}, and \textsc{Yan} \cite{yan03-1, yan03}. We  point out that in this connection  the differential inclusion method has been recently used in \textsc{De Lellis \& Sz\'ekelyhidi} \cite{delellis09} to study the Euler equations.
There are two  well-known  different approaches in solving the inclusion problem; however,   both derive basically  the same conclusions. The first method is the convex integration of \textsc{Gromov} \cite{gromov73}, elaborated  in \cite{muller01,muller96,muller99,muller03}. The other approach is the Baire's category method, exploited in \cite{dacorogna97,dacorogna99,kirchheim01,yan03-1,yan03}.  We explore  a simpler  Baire's category method based on the density argument to study differential inclusion (\ref{di-0}); our approach is quite different from that of \textsc{Zhang} \cite{zhang06} even for $n=1.$

Let us compare our result with that of \textsc{Ghisi \& Gobbino}  \cite{ghisi11}. Both  papers deal with radial solutions for the Perona-Malik equation in dimension $n\geq 2$. But our result can cover the case $n=1$, although there is an essential difference in the proof if $n\geq 2$. The paper \cite{ghisi11} presents radial \emph{classical} solutions over any annulus excluding the origin to avoid some technical difficulty due to the singularity of the corresponding one-dimensional equation at $s=0$, and it is remarked that excluding the origin may not be essential.  But we construct radial \emph{weak} solutions on a ball including the singularity at $s=0$ for the one-dimensional version, and we observe that the use of auxiliary functions in Section \ref{s4} is somehow \emph{optimal} as the $t$-derivatives of the functions may be very large if the positions $s$ of the functions are close enough to $s=0$, and a cutting line parallel to the $t$-axis for gluing may be very close to the $t$-axis according to the choice of initial data $u_0$. (See item (d) of Lemma \ref{4-16}.) The major difference between the two papers is in the admissible classes of initial data $u_0$ for solvability. In \cite{ghisi11}, the class of possible initial conditions for classical solvability is severely restricted due to the presence of backward (supercritical) region of $u_0$. One has much freedom in choosing the initial values in forward (subcritical) region of $u_0$, but then the initial values in backward region are \emph{determined} by the values in forward region. This phenomenon seems inevitable due to the inherent feature of the forward-backward radial problem. On the other hand, our result can give infinitely many radial weak solutions for any nonconstant smooth radial initial data $u_0$ whether it is transcritical or not. In fact, our result shows that, restricted to the nonconstant radially symmetric initial data, no matter it is the specially selected initial condition in \cite{ghisi11} or the initial condition which is all subcrtical (so the classical solution exists by \cite{kawohl98}), the problem (\ref{1-1}) will always have infinitely many (Lipschitz) radial  weak  solutions.  \\

The rest of this paper is organized as follows. In Section \ref{s2}, we introduce more notations and gather some of the ingredients needed to prove Theorem \ref{1-8}. A Baire's category method is introduced in Subsection \ref{s2-1} and  a classical result for  uniformly parabolic Neumann problems is included  in Subsection \ref{s2-2}  as  a building block that is to be modified by our approach. Section \ref{s2-3} contains the main setup of the problem (\ref{1-1}) into the framework of differential inclusion and the main density result, Theorem \ref{2-3-1}, which plays a pivot role in constructing a weak solution via Baire's method. Section \ref{s3} is devoted to the proof of Theorem \ref{1-8} based on Theorem \ref{2-3-1}. The construction of auxiliary functions needed for the proof of Theorem \ref{2-3-1} is given in Section \ref{s4}.  The proof of  Theorem \ref{2-3-1} is finally given  in  Section \ref{s5}.


\section{Notation and preliminaries}\label{s2}

We introduce some notations here. Let $N,n\in\NN$. For any measurable set $X\subset\RR^n$, $|X|$ denotes the Lebesgue measure of $X$. We denote by $\RR^{N\times n}$ the space of $N\times n$ real matrices, and for each $A=(a_{ij})\in\RR^{N\times n}$, we let $|A|$ be the Hilbert-Schmidt norm of $A$, that is,
$$
|A|:=\left(\sum_{i=1}^N\sum_{j=1}^{n} a_{ij}^2\right)^{1/2}.
$$
We let $O(n)$ denote the space of $n\times n$ orthogonal real matrices. For each $A\in\RR^{N\times n}$ and each $K\subset\RR^{N\times n}$, the distance from $A$ to the set $K$ is defined by
$$
\mathrm{dist}(A,K):=\inf_{B\in K}|A-B|.
$$
For $1\leq p\leq\infty$, let $W^{1,p}(\Omega;\RR^N)$ denote the usual Sobolev space of functions $u\in L^p(\Omega;\RR^N)$ whose first weak derivatives of each component exist and belong to $L^p(\Omega)$, where $\Omega\subset\RR^n$ is open. Also $W^{1,\infty}_0(\Omega;\RR^N):=W^{1,\infty}(\Omega;\RR^N)\cap W^{1,1}_0(\Omega;\RR^N)$, where $W^{1,1}_0(\Omega;\RR^N)$ is the closure of $C^\infty_0(\Omega;\RR^N)$ in $W^{1,1}(\Omega;\RR^N)$.

The following two lemmas are standard and used throughout this paper; see, e.g., \cite{dacorogna08,dacorogna99}.
\begin{lem}[Vitali Covering Lemma]
Let $\tilde\Omega$ and $\Omega$ be open sets in $\mathbb{R}^n$ with $\Omega$ bounded and $|\partial\Omega|=0$. Then for each $\epsilon>0$, there exist a sequence $\{x_j\}_{j\in\mathbb{N}}$ in $\mathbb{R}^n$ and a sequence $\{\epsilon_j\}_{j\in\mathbb{N}}$ of positive reals such that
\begin{equation*}
\left\{ \begin{array}{l}
  x_j+\epsilon_j\Omega\subset\tilde{\Omega}\,\,\,\,\,\,\textrm{and}\,\,\,\,\,\, \epsilon_j\leq\epsilon\,\,\,\,\,\, \forall j\in\mathbb{N},\\
  (x_j+\epsilon_j\Omega)\cap(x_k+\epsilon_k\Omega)=\emptyset\,\,\,\,\,\, \forall j,k\in\mathbb{N} \textrm{ with } j\neq k,\\
  |\Omega\setminus\cup_{j=1}^\infty(x_j+\epsilon_j\Omega)|=0.
\end{array} \right.
\end{equation*}
\end{lem}

\begin{lem}[Gluing lemma]
Let $\Omega$ be a bounded open set in $\mathbb{R}^n$, and let $\{\Omega_j\}_{j\in\mathbb{N}}$ be a sequence of disjoint open sets in $\Omega$. Let $u\in W^{1,\infty}(\Omega;\mathbb{R}^N)$, and let $u_j\in u+W^{1,\infty}_0(\Omega_j;\mathbb{R}^N)$ for each $j\in\mathbb{N}$. If $\sup_{j\in\mathbb{N}}\|u_j\|_{W^{1,\infty}(\Omega_j;\mathbb{R}^N)}<\infty$ and $\tilde{u}:=u\chi_{\Omega\setminus\cup_{j=1}^\infty \Omega_j}+\sum_{j=1}^\infty u_j\chi_{\Omega_j}$, then
$\tilde{u}\in u+W^{1,\infty}_0(\Omega;\mathbb{R}^N).$
\end{lem}

\subsection{A Baire's category method}\label{s2-1}

\begin{defi}[Baire-one map]
Let $X$ and $Y$ be metric spaces. Then $f:X\rightarrow Y$ is called a Baire-one map if it is pointwise limit of a sequence of continuous maps from $X$ into $Y$.
\end{defi}

The proofs of the next two results can be found in  \cite[Chapter 10]{dacorogna08}.

\begin{thm}[Baire's Category Theorem] \label{2-1-1}
Let $X$ and $Y$ be metric spaces with $X$ complete. If $f:X\rightarrow Y$ is a Baire-one map, then $\mathcal{D}_f$ is of the first category, where $\mathcal{D}_f$ is the set of points in $X$ at which $f$ is discontinuous. Therefore, the set $\mathcal{C}_f$ of points in $X$ at which $f$ is continuous, that is,  $\mathcal{C}_f:=X\setminus \mathcal{D}_f$,  is dense in $X$.
\end{thm}

\begin{prop}\label{2-1-2}
Let $N$ and $n$ be two positive integers. Let $U$ be a bounded open set in $\mathbb{R}^n$, and let $X \subset W^{1,\infty}(U;\mathbb{R}^N)$ be equipped with the $L^{\infty}(U;\mathbb{R}^N)$-metric. Then the gradient operator $$\nabla: X\rightarrow L^p(U;\mathbb{R}^{N\times n})$$ is a Baire-one map for every $p\in [1,\infty)$.
\end{prop}

Observe that if $X$ in Proposition \ref{2-1-2} is complete with respect to the $L^\infty$-metric, it follows from Theorem \ref{2-1-1} that the set $\mathcal{C}_\td$ of points of continuity for the gradient operator $\td$ is $L^\infty$-dense in $X$. In our application we take $p=1$ and $X$ to be the $L^\infty$-closure of the admissible class $\mathcal{P}^{n-1}_{\lambda,l_0}$ defined in Section \ref{s2-3} with $m=n-1$, so that $\mathcal{C}_\td$ is $L^\infty$-dense in $X$. This is a much shorter way to achieve the important principle that controlled $L^\infty$ convergence implies $W^{1,1}$ convergence, exlpored in \cite{muller01} by convex integration method. This explains that Baire's method is \emph{somehow} equivalent to the convex integration.

\subsection{Classical solution as building block}\label{s2-2}
We need the following result to build the \emph{nonempty} admissible class $\mathcal{P}^{n-1}_{\lambda,l_0}$ for the proof of Theorem \ref{1-8}.

\begin{thm}\label{2-2-1}
Let $u_0\in C^{3,\alpha}(\bar\Omega)$ be a radially symmetric function such that the compatibility condition holds:
\begin{displaymath}
\frac{\partial u_0}{\partial \mathbf{n}}(x)=0 \,\,\,\,\,\,\forall x\in\partial\Omega.
\end{displaymath}
Let $a^*\in C^{2,\alpha}([0,\infty))$ be positive on $[0,\infty)$. Define $\sigma^*(p):=a^*(p^2)p$ for every $p\in\RR$. Suppose that there exist two constants $C\geq c>0$ such that
\begin{equation}\label{para-0}
c\leq (\sigma^*)'(p)\leq C\,\,\,\,\,\,\forall p\geq 0.
\end{equation}
Then the Neumann problem
\begin{equation}\label{2-2-2}
\left\{ \begin{array}{ll}
  u^*_t=\mathrm{div}(a^*(|Du^*|^2)Du^*) & \textrm{in }\Omega_T \\
  \frac{\partial u^*}{\partial \mathbf{n}}=0 & \textrm{on }\partial\Omega\times(0,T) \\
  u^*(x,0)=u_0(x) & \textrm{for } x\in\Omega
\end{array} \right.
\end{equation}
has a unique solution $u^*\in C^{3+\alpha,1+\alpha/2}(\bar\Omega_T)$. Moreover, $u^*$ is radially symmetric in $\bar\Omega_T$, that is, for each $t\in[0,T]$,
$u^*(x,t)=u^*(y,t)$ whenever $x,y\in\bar\Omega,\,\,|x|=|y|,$
and we have the maximum principle:
\begin{displaymath}
\max_{\bar\Omega_T}|Du^*|=\max_{\bar\Omega}|Du_0|.
\end{displaymath}
\end{thm}

\begin{proof} By (\ref{para-0}) and the positivity of $a^*$, the problem (\ref{2-2-2}) is uniformly parabolic. Existence and uniqueness of  classical solution to problem (\ref{2-2-2}) are standard for parabolic equations \cite{ladyzenskaja67, lieberman96}. We only include a proof for the radial symmetry and maximum principle.  In the case $n=1$, the radial symmetry (i.e., $u^*(-x,t)=u^*(x,t)$) is easy and the maximum principle is also standard; so let us assume $n\ge 2.$   We first show that the solution $u^*$ is radially symmetric in $x$ on $\bar{\Omega}_T$. Suppose on the contrary that there exist two distinct points $x^0,y^0\in\Omega$ with $|x^0|=|y^0|$ and a time $t^0\in(0,T)$ such that
$
u^*(x^0,t^0)\neq u^*(y^0,t^0).
$
We can choose a  matrix $A\in O(n)$ such that $y^0=Ax^0$, where $x^0$, $y^0$ are regarded as column vectors. Define
$$
\tilde{u}^*(x,t):=u^*(Ax,t)\,\,\,\,\,\,\forall(x,t)\in\bar{\Omega}_T.
$$
Then it is straightforward to check  that $\tilde{u}^*\in C^{3+\alpha,1+\alpha/2}(\bar{\Omega}_T)$ solves the problem (\ref{2-2-2}). But
$$
\tilde{u}^*(x^0,t^0)=u^*(Ax^0,t^0)=u^*(y^0,t^0)\neq u^*(x^0,t^0),
$$
and this is a contradiction to the uniqueness of solution of (\ref{2-2-2}). Thus $u^*$ is radially symmetric in $\bar{\Omega}_T$. Note that $Du^*(0,t)=0$ for all $t\in[0,T]$ by the radial symmetry and differentiability of $u^*$ and that $Du^*(x,t)=0$ for every $(x,t)\in\partial \Omega\times[0,T]$ by the Neumann boundary condition and the radial symmetry of $u^*$. Next, we establish the maximum principle
\begin{equation}\label{6-2-0}
\max_{\bar\Omega_T}|Du^*|=\max_{\bar\Omega}|Du_0|.
\end{equation}
Let $v^*(s,t)=u^*(x,t)$, where $s=|x|.$ Then $|v^*_s(s,t)|=|Du^*(x,t)|$ with $s=|x|$ and hence $v^*_s(0,t)=v^*_s(R,t)=0$ for all $t\in [0,T].$ Similarly as in the introduction (or see (\ref{3-3-0}) below), the function $v^*$ solves the equation:
\[
v^*_t=(\sigma^*(v^*_s))_s+\sigma^*(v^*_s)\frac{n-1}{s}\,\,\,\,\,\,\textrm{in}\,\,(0,R)\times(0,T).
\]
Let $w^*=v^*_s.$ Then $w^*$ solves the following equation in $(0,R)\times(0,T)$
\begin{equation}\label{6-2}
\begin{cases} w^*_t=(\sigma^*) '(w^*) w^*_{ss} + (\sigma^*) ''(w^*)(w^*_{s})^2+ (\sigma^*) '(w^*)w^*_s \frac{n-1}{s} -\sigma^*(w^*)\frac{n-1}{s^2},\\
w^*(0,t)=w^*(R,t)=0\quad \forall \, t\in [0,T].
\end{cases}
\end{equation}
It is then  easy to show that
\[
\max_{[0,R]\times [0,T]} |w^*|=\max_{[0,R]} |w^*(\cdot,0)|.
\]
(The presence of the term $-\sigma^*(w^*)\frac{n-1}{s^2}$ in (\ref{6-2}) makes the proof much easier.) From this,  (\ref{6-2-0}) follows.
\end{proof}


\section{Basic setup and the density theorem}\label{s2-3}

In this section, we rephrase the problem (\ref{1-1}) into the frame work of partial differential inclusion (\ref{di-0}) with the set $\Sigma(s,v)$ replaced by a specific compact set $K^m_{\lambda,l_0}(s,v)$ with $m=n-1$, and then we present our main density result, Theorem \ref{2-3-1}, that is closely related to the reduction principle \cite{muller01} or relaxation property \cite{dacorogna08}. To this end, we set up the relevant definitions and prove some lemmas building up on the definitions that are to be used in the proofs of Theorem \ref{1-8} and Theorem \ref{2-3-1}. In doing so, we try to separate the arguments from these theorems to make our presentation as clear as possible.

\subsection{Several useful sets}
In what follows, let $\sigma(p)=a(p^2)p$  be  defined as above.
  It follows from (\ref{1-2}) that for each $q\in (0,\sigma(1))$, there are exactly two $p^+_q,\,p^-_q\in\RR$ such that
\begin{equation*}
 0<p^-_q<1<p^+_q, \quad
  \sigma(p^{\pm}_q)=q. \quad \mbox{(See Figure \ref{fig1}.)  }
\end{equation*}
For  each $\lambda>1$, let $\lambda^-:=p^-_{\sigma(\lambda)}$, and define the  sets
\begin{eqnarray}\label{2-3-3}
\tilde{K}_\lambda &:=& \{(p,\sigma(p))\in\RR^2:|p|\leq\lambda\},\nonumber\\
\tilde{U}^+_\lambda &:=& \{(p,q)\in\RR^2:\sigma(\lambda)<q<\sigma(1),\,p^-_q<p<p^+_q\},\\
\tilde{U}^-_\lambda &:=& \{(p,q)\in\RR^2:(-p,-q)\in\tilde{U}^+_\lambda\}.\nonumber
\end{eqnarray}
(See Figure \ref{fig2}.)
We begin with  the following technical lemma whose proof can be found in   \cite[Lemma 3.1]{zhang06}.

\begin{lem}\label{2-3-2}
Let $\lambda>1$ and $\lambda^-<M<\lambda$. Then, there exists an odd function $\sigma^*\in C^{2,\alpha}(\RR)$ satisfying the following:
\begin{itemize}
\item[(a)] $\sigma^*(p)=\sigma(p)$ for $0\leq p\leq\lambda^-$, $\sigma^*(p)<\sigma(M)$ for $\lambda^-<p\leq M$, and
\item[(b)] there exist two constants $C\geq c>0$ such that
    $$
    c\leq(\sigma^*)'(p)\leq C\,\,\textrm{ for every }p\geq 0.
    $$
\end{itemize}
\end{lem}
We remark that the function $\sigma^*$ depends on $\lambda$ and $M$.

 Let $m\geq 0$ be a fixed integer and $\lambda>1$ in the rest of this section. Here let us keep in mind that $m=n-1$ in our application, where $n$ is the space dimension in Theorem \ref{1-8}. For each $s>0$, define
\begin{eqnarray}\label{2-3-4}
\tilde{K}^m_\lambda(s) &:=& \{(p,s^mq)\in\RR^2:(p,q)\in\tilde{K}_\lambda\},\nonumber \\
\tilde{U}^m_\lambda(s) &:=& \{(p,s^mq)\in\RR^2:(p,q)\in\tilde{U}^+_\lambda\cup\tilde{U}^-_\lambda\},\\
I^m_\lambda(s,p) &:=& \left\{ \begin{array}{ll}
                            (s^m\sigma(\lambda),s^m\sigma(p))\subset\RR & \textrm{if }\lambda^-<p<\lambda \\
                            (s^m\sigma(p),-s^m\sigma(\lambda))\subset\RR & \textrm{if }-\lambda<p<-\lambda^-.
                          \end{array} \nonumber
 \right.
\end{eqnarray}
Given any $l_0>0$, for $s>0$ and $v\in\RR$, define the sets in $\RR^{2\times 2}$:
\begin{eqnarray}
K^m_{\lambda,l_0}(s,v) &:=& \left\{\left(
                           \begin{array}{cc}
                             p & l \\
                             s^m v & s^m q \\
                           \end{array}
                         \right)
\in\RR^{2\times2}:(p,q)\in\tilde{K}_\lambda,\,|l|\leq l_0 \right\}, \label{2-3-4-k}\\
U^m_{\lambda,l_0}(s,v) &:=& \left\{\left(
                           \begin{array}{cc}
                             p & l \\
                             s^m v & s^m q \\
                           \end{array}
                         \right)
\in\RR^{2\times2}:(p,q)\in\tilde{U}^+_\lambda\cup\tilde{U}^-_\lambda,\,|l|< l_0 \right\}.\label{2-3-4-u}
\end{eqnarray}
We also let $l_0>0$ be fixed throughout the rest of this section.

\begin{figure}[h]
\begin{center}
\includegraphics[scale=0.6]{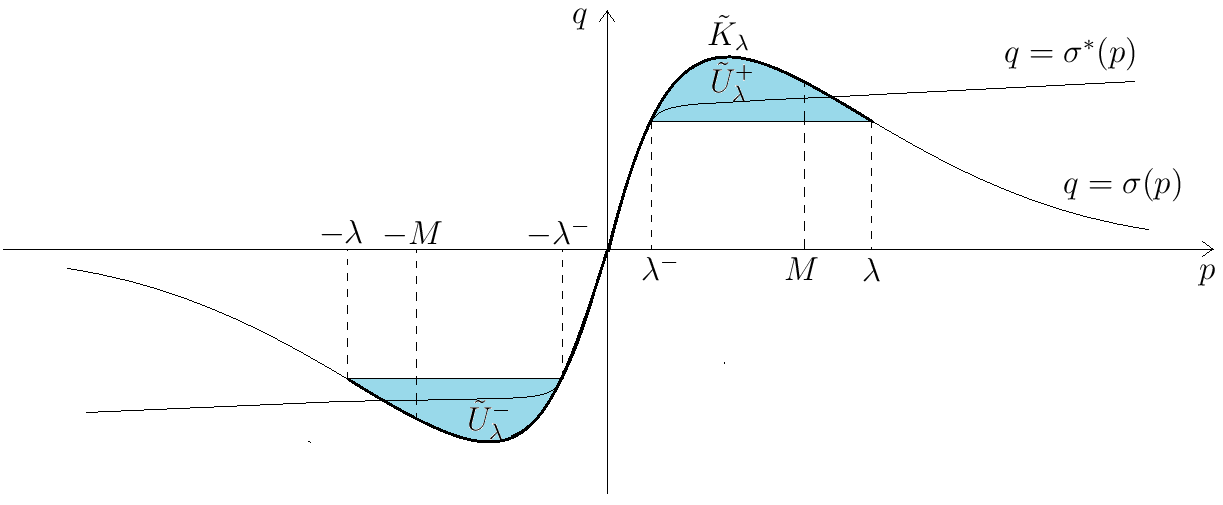}
\end{center}
\caption{The graph of a new function $q=\sigma^*(p)$ from Lemma \ref{2-3-2}}
\label{fig2}
\end{figure}

\subsection{Properties of some distance functions}
The following four lemmas are basically on the reformulations of some (inhomogeneous) distance functions into simpler expressions that we can easily manage for the proof of the density result, Theorem \ref{2-3-1}.
\begin{lem}\label{2-3-9}
Let $s>0$. Then for each $(p,q')\in\RR^2$,
$$
(p,q')\in\tilde{U}^m_\lambda(s)\,\,\,\textrm{ if and only if }\,\,\,p\in(-\lambda,-\lambda^-)\cup(\lambda^-,\lambda),\,q'\in I^m_\lambda(s,p).
$$
\end{lem}
\begin{proof}
Let $(p,q')\in\RR^2$. Assume $(p,q')\in\tilde{U}^m_\lambda(s)$. Put $q:=s^{-m}q'$. Then $(p,q)\in\tilde{U}^+_\lambda\cup\tilde{U}^-_\lambda$. If $(p,q)\in\tilde{U}^+_\lambda$, then $\lambda^-<p<\lambda$ and $\sigma(\lambda)<q<\sigma(p)$, and so $s^m\sigma(\lambda)<q'=s^mq<s^m\sigma(p)$, i.e., $q'\in I^m_\lambda(s,p).$ If $(p,q)\in\tilde{U}^-_\lambda$, then $-\lambda<p<-\lambda^-$ and $\sigma(p)<q<-\sigma(\lambda)$, and so $s^m\sigma(p)<q'=s^mq<-s^m\sigma(\lambda)$, i.e., $q'\in I^m_\lambda(s,p).$ The converse statement is also easy, and we omit it.
\end{proof}

For each $v'\in\RR$, define
\[
W_{v'}:=\left\{\left(
              \begin{array}{cc}
                a & b \\
                v' & d \\
              \end{array}
            \right)\in\RR^{2\times2}:a,b,d\in\RR\right\}.
\]
If $K\subset W_{v'}$, let $\partial|_{W_{v'}}K$ denote the relative boundary of $K$ in $W_{v'}$.
Let $W:=W_0$, and let $P_W$ be the projection of $\RR^{2\times2}$ onto $W$, that is,
$$
P_W\left(\left(
     \begin{array}{cc}
       a & b \\
       c & d \\
     \end{array}
   \right)\right)=\left(
     \begin{array}{cc}
       a & b \\
       0 & d \\
     \end{array}
   \right)\quad \forall \left(
     \begin{array}{cc}
       a & b \\
       c & d \\
     \end{array}
   \right)\in\RR^{2\times 2}.
$$
For example, $K^m_{\lambda,l_0}(s,v),\,U^m_{\lambda,l_0}(s,v)\subset W_{s^mv}$, where $s>0$ and $v\in\RR$.

\begin{lem}\label{2-3-7}
Let $s>0$ and $v\in\RR$, and let
$$
A=\left(
     \begin{array}{cc}
       \tilde{p} & \tilde{l} \\
       s^mv & \tilde{q}' \\
     \end{array}
   \right)\in\RR^{2\times 2}.
$$
Then
$$
\mathrm{dist}(A,K^m_{\lambda,l_0}(s,v)\cup \partial|_{W_{s^mv}}U^m_{\lambda,l_0}(s,v))=\mathrm{dist}(P_W(A),K^m_{\lambda,l_0}(s,0)\cup \partial|_{W}U^m_{\lambda,l_0}(s,0)).
$$
\end{lem}

\begin{proof}
Observe that
$$
K^m_{\lambda,l_0}(s,0)\cup \partial|_{W}U^m_{\lambda,l_0}(s,0)=\{P_W(X)\in W:X\in K^m_{\lambda,l_0}(s,v)\cup \partial|_{W_{s^mv}}U^m_{\lambda,l_0}(s,v)\}
$$
and that if $X,\,Y\in\RR^{2\times2}$ are such that $X_{21}=Y_{21}$, then $|P_W(X)-P_W(Y)|=|X-Y|$.
Thus
$$
\mathrm{dist}(A,K^m_{\lambda,l_0}(s,v)\cup \partial|_{W_{s^mv}}U^m_{\lambda,l_0}(s,v))
$$
$$=
\inf\{|A-X|:X\in K^m_{\lambda,l_0}(s,v)\cup \partial|_{W_{s^mv}}U^m_{\lambda,l_0}(s,v)\}
$$
$$
=\inf\{|P_W(A)-P_W(X)|:X\in K^m_{\lambda,l_0}(s,v)\cup \partial|_{W_{s^mv}}U^m_{\lambda,l_0}(s,v)\}
$$
$$
=\mathrm{dist}(P_W(A),K^m_{\lambda,l_0}(s,0)\cup \partial|_{W}U^m_{\lambda,l_0}(s,0)).
$$
\end{proof}

\begin{lem}\label{2-3-8}
Let $F\subset\RR_+\times\RR$ be a compact set, where $\RR_+:=\{s\in\RR:s>0\}$.
If $T:F\RA W$ is a continuous mapping, then the mapping $d:F\RA[0,\infty)$, defined by
$$
d(s,t):=\mathrm{dist}(T(s,t),K^m_{\lambda,l_0}(s,0)\cup \partial|_{W}U^m_{\lambda,l_0}(s,0))\,\,\,\forall (s,t)\in F,
$$
is also continuous.
\end{lem}

\begin{proof}
Let $\epsilon>0$. By the uniform continuity of $T$ on $F$, there exists a $\delta>0$ such that
$$
|T(s_1,t_1)-T(s_2,t_2)|\leq\frac{\epsilon}{2}
$$
whenever $(s_1,t_1),(s_2,t_2)\in F$, $|(s_1,t_1)-(s_2,t_2)|\leq\delta$.
Fix any two $(s_1,t_1),(s_2,t_2)\in F$ with $|(s_1,t_1)-(s_2,t_2)|\leq\delta$.
Since $K^m_{\lambda,l_0}(s_1,0)\cup \partial|_{W}U^m_{\lambda,l_0}(s_1,0)$ is compact, we can choose a matrix $\left(
                                                                             \begin{array}{cc}
                                                                               \tilde{p}_1 & \tilde{l}_1 \\
                                                                               0 & \tilde{q}_1' \\
                                                                             \end{array}
                                                                           \right) $ in this compact set so that
$$
d(s_1,t_1)=\left|T(s_1,t_1)-\left(
                                                                             \begin{array}{cc}
                                                                               \tilde{p}_1 & \tilde{l}_1 \\
                                                                               0 & \tilde{q}_1' \\
                                                                             \end{array}
                                                                           \right) \right|.
$$
Put $\tilde{q}_1:=(s_1)^{-m}\tilde{q}_1'$. Then $(\tilde{p}_1,\tilde{q}_1)\in \tilde{K}_\lambda\cup(\overline{\tilde{U}^+_\lambda\cup\tilde{U}^-_\lambda})$ if $\tilde{l}_1\in\{l_0,-l_0\}$ or $(\tilde{p}_1,\tilde{q}_1)\in \tilde{K}_\lambda\cup(\partial \tilde{U}^+_\lambda\cup \partial\tilde{U}^-_\lambda)$ if $\tilde{l}_1\in(-l_0,l_0)$. So we have
$$
\left(
                                                                             \begin{array}{cc}
                                                                               \tilde{p}_1 & \tilde{l}_1 \\
                                                                               0 & (s_2)^m\tilde{q}_1 \\
                                                                             \end{array}
                                                                           \right) \in
                                                                           K^m_{\lambda,l_0}(s_2,0)\cup \partial|_{W}U^m_{\lambda,l_0}(s_2,0).
$$
Note that
\begin{eqnarray*}
d(s_2,t_2) &\leq& \left|T(s_2,t_2)-\left(
                                                                             \begin{array}{cc}
                                                                               \tilde{p}_1 & \tilde{l}_1 \\
                                                                               0 & (s_2)^m\tilde{q}_1 \\
                                                                             \end{array}
                                                                           \right) \right|\\
&\leq& |T(s_2,t_2)-T(s_1,t_1)|+\left|T(s_1,t_1)-\left(
                                                                             \begin{array}{cc}
                                                                               \tilde{p}_1 & \tilde{l}_1 \\
                                                                               0 & (s_1)^m\tilde{q}_1 \\
                                                                             \end{array}
                                                                           \right) \right|\\
& & +\left|\left(
                                                                             \begin{array}{cc}
                                                                               \tilde{p}_1 & \tilde{l}_1 \\
                                                                               0 & (s_1)^m\tilde{q}_1 \\
                                                                             \end{array}
                                                                           \right)-\left(
                                                                             \begin{array}{cc}
                                                                               \tilde{p}_1 & \tilde{l}_1 \\
                                                                               0 & (s_2)^m\tilde{q}_1 \\
                                                                             \end{array}
                                                                           \right) \right|,
\end{eqnarray*}
and so
\begin{eqnarray*}
d(s_2,t_2)-d(s_1,t_1)
&\leq& |T(s_2,t_2)-T(s_1,t_1)|+\left|\left(
                                                                             \begin{array}{cc}
                                                                               \tilde{p}_1 & \tilde{l}_1 \\
                                                                               0 & (s_1)^m\tilde{q}_1 \\
                                                                             \end{array}
                                                                           \right)-\left(
                                                                             \begin{array}{cc}
                                                                               \tilde{p}_1 & \tilde{l}_1 \\
                                                                               0 & (s_2)^m\tilde{q}_1 \\
                                                                             \end{array}
                                                                           \right) \right|\\
&\leq& \frac{\epsilon}{2}+\sigma(1)|(s_1)^m-(s_2)^m|.
\end{eqnarray*}
Let $I\subset\RR_+$ be a compact interval with $\{s\in\RR_+:(s,t)\in F\}\subset I$. Then the mapping $s\mapsto s^m$ is uniformly continuous on $I$, so that there exists a $\delta'>0$ such that
$$
s_1,s_2\in I,\,|s_1-s_2|\leq\delta'\Rightarrow |(s_1)^m-(s_2)^m|\leq\frac{\epsilon}{2\sigma(1)}.
$$
Thus if $(s_1,t_1),(s_2,t_2)\in F$ and $|(s_1,t_1)-(s_2,t_2)|\leq\min\{\delta,\delta'\}$, then
$$
d(s_2,t_2)-d(s_1,t_1)\leq\epsilon.
$$
Changing the roles of $(s_1,t_1)$ and $(s_2,t_2)$ and combining the results, we obtain the continuity of the mapping $d$ on $F$.
\end{proof}

\begin{lem}\label{2-3-6}
Let $s>0$ and $v\in\RR$, and let
$$
A:=\left(
     \begin{array}{cc}
       \tilde{p} & \tilde{l} \\
       s^mv & \tilde{q}' \\
     \end{array}
   \right)\in\RR^{2\times 2}
$$
be such that $|\tilde{l}|\leq l_0$. Then
$$
\mathrm{dist}(A,K^m_{\lambda,l_0}(s,v))=\mathrm{dist}((\tilde{p},\tilde{q}'),\tilde{K}^m_\lambda(s)).
$$
\end{lem}

\begin{proof}
Choose any $(p,q)\in\tilde{K}_\lambda$. Then
$$
\mathrm{dist}(A,K^m_{\lambda,l_0}(s,v))\leq \left|\left(
     \begin{array}{cc}
       \tilde{p} & \tilde{l} \\
       s^mv & \tilde{q}' \\
     \end{array}
   \right)-\left(
     \begin{array}{cc}
       p & \tilde{l} \\
       s^mv & s^mq \\
     \end{array}
   \right)\right|=|(\tilde{p},\tilde{q}')-(p,s^mq)|.
$$
Taking an infimum on $(p,q)\in\tilde{K}_\lambda$, we have
$$
\mathrm{dist}(A,K^m_{\lambda,l_0}(s,v))\leq\mathrm{dist}((\tilde{p},\tilde{q}'),\tilde{K}^m_\lambda(s)).
$$
To show the reverse inequality, choose any $(p,q)\in\tilde{K}_\lambda$ and any $l\in\RR$ with $|l|\leq l_0$. Then
$$
\mathrm{dist}((\tilde{p},\tilde{q}'),\tilde{K}^m_\lambda(s))\leq |(\tilde{p},\tilde{q}')-(p,s^mq)|=\left|\left(
     \begin{array}{cc}
       \tilde{p} & \tilde{l} \\
       s^mv & \tilde{q}' \\
     \end{array}
   \right)-\left(
     \begin{array}{cc}
       p & \tilde{l} \\
       s^mv & s^mq \\
     \end{array}
   \right)\right|
$$
$$
\leq \left|\left(
     \begin{array}{cc}
       \tilde{p} & \tilde{l} \\
       s^mv & \tilde{q}' \\
     \end{array}
   \right)-\left(
     \begin{array}{cc}
       p & l \\
       s^mv & s^mq \\
     \end{array}
   \right)\right|,
$$
so that taking an infimum on $(p,q,l)\in\tilde{K}_\lambda\times[-l_0,l_0]$, we have
$$
\mathrm{dist}((\tilde{p},\tilde{q}'),\tilde{K}^m_\lambda(s))\leq\mathrm{dist}(A,K^m_{\lambda,l_0}(s,v)).
$$
Thus the lemma is proved. \end{proof}

\subsection{Admissible class and the density theorem}\label{s3-3}
Let $J:=(0,R)\subset\RR$, and $J_T:=J\times(0,T)\subset\RR^2$. Fix a $\delta_0\in\RR$ with $0<\delta_0<R/2$, and put $J^*_T:=(\delta_0,R-\delta_0)\times(0,T)\subset J_T$.

Let $\Phi^*=(v^*,\vp^*)\in W^{1,\infty}(J^*_T;\RR^2)$ be a given piecewise $C^1$ function in $J^*_T$. We define the admissible class needed for  later construction of the weak solutions as follows:
\begin{equation}\label{2-3-5}
\PP^m_{\lambda,l_0}:=\left\{\Phi\in \Phi^*+ W^{1,\infty}_0(J^*_T;\RR^2):
\left. \begin{array}{l}
  \Phi=(v,\vp) \textrm{ is piecewise $C^1$ in } J^*_T, \\
  \nabla\Phi(s,t)=\left(
                           \begin{array}{cc}
                             v_s(s,t) & v_t(s,t) \\
                             \vp_s(s,t) & \vp_t(s,t) \\
                           \end{array}
                         \right)
   \\
  \in K^m_{\lambda,l_0}(s,v(s,t))\cup U^m_{\lambda,l_0}(s,v(s,t)) \\
  \textrm{for a.e. } (s,t)\in J^*_T
\end{array} \right. \right\}.
\end{equation}
Note that this set may be  empty; but in our application below, we will define a function $\Phi^*$ so that this class  $\PP^m_{\lambda,l_0}$ is nonempty.

We are now in a position to state the following main density result, whose proof will be postponed to   Section \ref{s5}.
\begin{thm}[Density Theorem] \label{2-3-1}
For each $\epsilon>0$, the set
$$
\PP^m_{\lambda,l_0,\epsilon}:=\left\{\Phi\in\PP^m_{\lambda,l_0}:
\int_{J^*_T}\mathrm{dist}(\td\Phi(s,t),K^m_{\lambda,l_0}(s,v(s,t)))dsdt\leq\epsilon|J^*_T|\right\}
$$
is dense in $\PP^m_{\lambda,l_0}$ with respect to the $L^\infty(J^*_T;\RR^2)$-metric.
\end{thm}


\section{Proof of Theorem \ref{1-8}}\label{s3}

In this section we  aim to prove Theorem \ref{1-8} based on the density theorem, Theorem \ref{2-3-1}. To this end, we assume $a,\sigma$ and $u_0$ are functions given as above.

\subsection{The modified parabolic problem} Let $J$, $J_T$ be defined as in Subsection \ref{s3-3}.  Since $u_0$ is radial, let $u_0(x)=v_0(|x|)$ for a function $v_0\in C^{3,\alpha}(\bar{J})$, and hence
\begin{eqnarray}\label{3-1}
\max_{\bar{J}}|v_0'|=\max_{\bar\Omega}|Du_0|=M>0.
\end{eqnarray}
Fix any $\epsilon>0$. We define a number $\lambda>1$ as follows: if $M\ge 1$, let $\lambda =M+\epsilon$; if $0<M<1$, let $\lambda\gg 1$ be such that $\sigma(\lambda)<\sigma(M).$  Then we always have that $\lambda^-<M<\lambda.$

With the choice of $M$ and $\lambda$, let  $\sigma^*$ be a function that can be determined by  Lemma \ref{2-3-2}. Define $a^*(p):=\sigma^*(\sqrt p)/\sqrt{p}$ for each $p>0$. Then $a^*(p)=\sigma^*(\sqrt p)/\sqrt{p}=\sigma(\sqrt p)/\sqrt{p}=a(p)$ for every $p\in(0,(\lambda^-)^2]$. Since $a\in C^{2,\alpha}([0,\infty))$, we also have $a^*\in C^{2,\alpha}([0,\infty))$. Also the functions $a^*$ and $\sigma^*$ satisfy the hypotheses in Theorem \ref{2-2-1}. Therefore, for the given initial condition $u_0$, the problem (\ref{2-2-2}) has a unique radial solution $u^*\in C^{3+\alpha,1+\alpha/2}(\bar\Omega_T)$ with the maximum principle
\begin{equation}\label{3-2}
\max_{\bar\Omega_T}|Du^*|=\max_{\bar\Omega}|Du_0|=M>0.
\end{equation}
Let $u^*(x,t)=v^*(|x|,t)$ for a function $v^*:\bar{J}_T\RA\RR$.  Then $v^*\in C^{3+\alpha,1+\frac{\alpha}{2}}(J_T)$.
Let $(x,t)\in\{\Omega\setminus\{0\}\}\times(0,T)$. For each $i\in\{1,\ldots,n\}$,
$$
\partial_i u^*(x,t):=\partial_{x_i} u^* (x,t)=v^*_s(|x|,t)\frac{x_i}{|x|}.
$$
So $Du^*(x,t)=v^*_s(|x|,t)\frac{x}{|x|}$, and hence
$$
a^*(|Du^*(x,t)|^2)Du^*(x,t)=a^*(v_s^*(|x|,t)^2)v^*_s(|x|,t)\frac{x}{|x|}.
$$
Taking divergence on both sides, we obtain
\begin{eqnarray*}
\mathrm{div}(a^*(|Du^*(x,t)|^2)Du^*(x,t))&=&(a^*(v^*_s(s,t)^2)v^*_s(s,t))_s\big{|}_{s=|x|}\\
&&+a^*(v^*_s(|x|,t)^2)v^*_s(|x|,t)\frac{n-1}{|x|}.
\end{eqnarray*}
Since $u^*_t(x,t)=v^*_t(|x|,t)$, we thus have
\begin{equation}\label{3-3-0}
v^*_t(s,t)=(a^*(v^*_s(s,t)^2)v^*_s(s,t))_s+a^*(v^*_s(s,t)^2)v^*_s(s,t)\frac{n-1}{s}.\quad (s=|x|)
\end{equation}
In summary, $v^*\in C^{3+\alpha,1+\alpha/2}(\bar{J}_T)$ solves the following problem:
\begin{equation}\label{3-3}
\left\{ \begin{array}{ll}
  (s^{n-1}v^*(s,t))_t=(s^{n-1}a^*(v^*_s(s,t)^2)v^*_s(s,t))_s & \textrm{for }(s,t)\in J_T \\
  v^*_s(0,t)=v^*_s(R,t)=0 & \textrm{for }t\in[0,T] \\
  v^*(s,0)=v_0(s) & \textrm{for } x\in \bar{J},
\end{array} \right.
\end{equation}
where
\begin{equation}\label{cont2}
\max_{\bar{J}_T}|v^*_s|=\max_{\bar{J}}|v_0'|=M.
\end{equation}

The uniform continuity of $v^*_s$ on $\bar{J}_T$ and the second of (\ref{3-3}) imply that there exists a $\delta_0\in(0,R/2)$ such that
\begin{eqnarray}\label{3-5}
\max_{\{[0,\delta_0]\cup[R-\delta_0,R]\}\times[0,T]}|v^*_s|\leq \lambda^-.
\end{eqnarray}
With this $\delta_0$, let $J_T^*$ be defined as in Subsection \ref{s3-3}.

\subsection{The starting function $\Phi^*$} We define $\Phi^*:=(v^*,\vp^*),$ where  $\vp^*:\bar{J}_T\RA\RR$ is given  by
$$
\vp^*(s,t):=\int_0^s w^{n-1}v^*(w,t)dw\,\,\textrm{ for every }(s,t)\in\bar{J}_T.
$$
Then $\vp^*\in C^{3+\alpha,1+\alpha/2}(\bar{J}_T)$, and
\begin{eqnarray}\label{3-11}
\vp^*_s(s,t)&=&s^{n-1}v^*(s,t),\nonumber\\
\vp^*_t(s,t)&=&\int_0^s w^{n-1}v^*_t(w,t)dw\nonumber\\
&=&\int^s_0(w^{n-1}a^*(v^*_s(w,t)^2)v^*_s(w,t))_wdw\,\,\,\,\,\,\textrm{(by (\ref{3-3}))}\\
&=&s^{n-1}a^*(v^*_s(s,t)^2)v^*_s(s,t)\nonumber
\end{eqnarray}
for every $(s,t)\in J_T$.
So $\Phi^*=(v^*,\vp^*)\in C^{3+\alpha,1+\alpha/2}(\bar{J}_T;\RR^2)$, and
$$
\td\Phi^*(s,t)=\left(
                 \begin{array}{cc}
                   v^*_s(s,t) & v^*_t(s,t) \\
                   \vp^*_s(s,t) & \vp^*_t(s,t) \\
                 \end{array}
               \right)
               =\left(
                  \begin{array}{cc}
                    v^*_s(s,t) & v^*_t(s,t) \\
                    s^{n-1}v^*(s,t) & s^{n-1}a^*(v^*_s(s,t)^2)v^*_s(s,t) \\
                  \end{array}
                \right).
$$

Put  $l_0:=\max_{\bar{J}_T}|v^*_t|+1>0$. Let $\tilde{K}_\lambda$ and $\tilde{U}_\lambda^\pm$ be defined as in (\ref{2-3-3}). For each $(s,t)\in J_T$, since $|v^*_s(s,t)|\leq M$, it follows from Lemma \ref{2-3-2} that
$$
(v^*_s(s,t),a^*(v^*_s(s,t)^2)v^*_s(s,t))=(v^*_s(s,t),\sigma^*(v^*_s(s,t)))\in\tilde{K}_\lambda\cup\tilde{U}^+_\lambda\cup\tilde{U}^-_\lambda
$$
and that
\begin{equation*}
(v^*_s(s,t),a^*(v^*_s(s,t)^2)v^*_s(s,t))=(v^*_s(s,t),\sigma(v^*_s(s,t)))\in\tilde{K}_\lambda\,\,\,\,\,\,\textrm{(by (\ref{3-5}))}
\end{equation*}
if $(s,t)\in J_T\setminus J^*_T$.
Hence
\begin{equation}\label{3-9}
\begin{cases} \td\Phi^*(s,t)\in K^{n-1}_{\lambda,l_0}(s,v^*(s,t))\cup U^{n-1}_{\lambda,l_0}(s,v^*(s,t))\ & \forall\,  (s,t)\in J_T,\\
\td\Phi^*(s,t)\in K^{n-1}_{\lambda,l_0}(s,v^*(s,t)) &\forall\,  (s,t)\in J_T\setminus J^*_T,
\end{cases}
\end{equation}
 where the sets $K^{n-1}_{\lambda,l_0}(s,v)$ and $U^{n-1}_{\lambda,l_0}(s,v)$ are defined as in (\ref{2-3-4-k}) and (\ref{2-3-4-u}) with $m=n-1$.

We now define the admissible class $\PP^{n-1}_{\lambda,l_0}$ by using this function $\Phi^*$ on $J_T^*$ as in (\ref{2-3-5}) with $m=n-1$.  Then clearly,
$$
\Phi^*\in\PP^{n-1}_{\lambda,l_0}\neq \emptyset.
$$

\subsection{The Baire category method} Let $X$ denote the closure of $\PP^{n-1}_{\lambda,l_0}$ in the space $L^\infty(J^*_T;\RR^2)$. Since the sets $K^{n-1}_{\lambda,l_0}(s,v)$ and $U^{n-1}_{\lambda,l_0}(s,v)$ are bounded, it is easily checked that
$$
\PP^{n-1}_{\lambda,l_0}\subset X\subset \Phi^*+W^{1,\infty}_0(J^*_T;\RR^2).
$$
Proposition \ref{2-1-2} shows that the gradient operator $\td:X\RA L^1(J^*_T;\RR^{2\times 2})$ is a Baire-one map, and so the set $\mathcal{C}_\td$ of points in $X$ at which the map $\td$ is continuous is dense in $X$ by Theorem \ref{2-1-1}. So we have $\mathcal{C}_\td\neq\emptyset$, since $X\neq\emptyset$. Later we show that $\mathcal{C}_\nabla$ is actually  an infinite set.
But first we elaborate on how  the density theorem (Theorem \ref{2-3-1}) guarantees that every function in $\mathcal{C}_\nabla$ provides  us a solution to Problem (\ref{1-1}).

Let $\Phi=(v,\vp)\in\mathcal{C}_\td\subset X.$  Let $k\in\NN$. By the definition of $X$, we can choose a $\tilde{\Phi}_k\in\PP^{n-1}_{\lambda,l_0}$ so that
$$
||\Phi-\tilde{\Phi}_k||_{L^\infty}\leq\frac{1}{k}.
$$
By the density theorem, Theorem \ref{2-3-1}, we can choose a function $\Phi_k=(v_k,\vp_k)\in\PP^{n-1}_{\lambda,l_0,1/k}$ so that
$$
||\tilde{\Phi}_k-\Phi_k||_{L^\infty}\leq\frac{1}{k}.
$$
Combining these two inequalities, we have
$$
||\Phi-\Phi_k||_{L^\infty}\leq\frac{2}{k}\RA 0\,\,\textrm{ as } k\RA\infty.
$$
Since the map $\td$ is continuous at $\Phi$, we thus have
$$
\td\Phi_k\RA\td\Phi\,\,\textrm{ in } L^1(J^*_T;\RR^{2\times 2})\,\,\textrm{ as }k\RA\infty.
$$
Upon passing to a subsequence (we do not relabel), we can assume that
\begin{equation}\label{pp-conv}
\td\Phi_k(s,t)\RA\td\Phi(s,t)\,\,\textrm{ in } \RR^{2\times 2}\,\,\textrm{ as }k\RA\infty,\,\,\textrm{ for a.e. }(s,t)\in J^*_T.
\end{equation}
Since $\Phi_k\in\mathcal{P}^{n-1}_{\lambda,l_0,1/k}$, it follows from Lemma \ref{2-3-6} that
\[
\begin{split} & \int_{J^*_T}\mathrm{dist}(((v_k)_s(s,t),(\vp_k)_t(s,t)),\tilde{K}^{n-1}_\lambda(s))dsdt  =
\\
&\int_{J^*_T}\mathrm{dist}(\td\Phi_k(s,t),K^{n-1}_{\lambda,l_0}(s,v_k(s,t)))dsdt \leq \frac{|J^*_T|}{k}\quad \forall\, k\in\NN.
\end{split}
\]
Applying Fatou's lemma to this inequality with (\ref{pp-conv}), we obtain
$$
\int_{J^*_T} \mathrm{dist}((v_s(s,t),\vp_t(s,t)),\tilde{K}^{n-1}_\lambda(s))\,dsdt =0.
$$
Since $\tilde{K}^{n-1}_\lambda(s)$ is closed in $\RR^2$ for each $s>0$, it follows that
\begin{equation}\label{3-6}
(v_s(s,t),\vp_t(s,t))\in\tilde{K}^{n-1}_\lambda(s)\,\,\textrm{ for a.e. }(s,t)\in J^*_T.
\end{equation}
Moreover, for each $k\in\NN$, we have
$$
|(v_k)_t(s,t)|\leq l_0,\,\,\,(\vp_k)_s(s,t)=s^{n-1}v_k(s,t) \,\,\textrm{ for a.e. } (s,t)\in J^*_T,
$$
so that letting $k\RA\infty$, it follows that
\begin{equation}\label{3-7}
|v_t(s,t)|\leq l_0,\,\,\vp_s(s,t)=s^{n-1}v(s,t)\,\,\textrm{ for a.e. } (s,t)\in J^*_T.
\end{equation}
Combining (\ref{3-6}) and (\ref{3-7}), we have
$$
\td\Phi(s,t)\in K^{n-1}_{\lambda,l_0}(s,v(s,t))\,\,\textrm{ for a.e. } (s,t)\in J^*_T.
$$
Since $\Phi\in\Phi^*+W^{1,\infty}_0(J^*_T;\RR^2)$, we can extend  $\Phi$ from $J^*_T$ to $J_T$ by setting
\begin{equation}\label{3-8}
\Phi:=\Phi^*\,\,\textrm{ on }J_T\setminus J^*_T.
\end{equation}
Then it follows that $\Phi\in\Phi^*+W^{1,\infty}_0(J_T;\RR^2)$ and $\Phi\equiv\Phi^*$ on $\overline{J_T\setminus J^*_T}$, where we still write $\Phi=(v,\vp)$ on $\bar{J}_T$. Observe now that by (\ref{3-9}),
\begin{equation}\label{3-10}
\td\Phi(s,t)\in K^{n-1}_{\lambda,l_0}(s,v(s,t))\,\,\textrm{ for a.e. } (s,t)\in J_T.
\end{equation}

Define
\begin{equation}\label{def-u}
u(x,t):=v(|x|,t),\quad \psi(x,t):=\vp(|x|,t)\quad \forall\, (x,t)\in \bar\Omega_T.
\end{equation}
By   (\ref{3-10})  and (\ref{def-u}), we have
\begin{equation}\label{pro-1}
Du(x,t)=v_s(s,t)\frac{x}{s},\quad D\psi(x,t)=\vp_s(s,t)\frac{x}{s}=|x|^{n-2}u(x,t) x, \quad s=|x|\neq 0.
\end{equation}
Since $(v,\vp)=(v^*,\vp^*)$ on $\overline{J_T\setminus J_T^*}$, it is guaranteed from the definition of $\vp^*$, (\ref{3-3}), and (\ref{3-11}) that  for all $t\in [0,T]$,
\begin{equation}\label{pro-2}
\vp(0,t)=0,\quad \vp(R,t)=\vp(R,0)=\int_0^R w^{n-1}v_0(w)\,dw.
\end{equation}

We now prove the following result.

\begin{thm} The function $u$ defined above solves problem (\ref{1-1}) in the sense that, for every $\xi\in C^1(\bar\Omega_T)$,
\begin{equation}\label{strong}
\int_\Omega \left ( u(x,T)\xi(x,T)  - u_0(x)\xi(x,0)\right ) dx =\int_{\Omega_T} (u \xi_t - a(|Du|^2)Du\cdot D\xi) \,dxdt.
\end{equation}
\end{thm}
\begin{proof} It is sufficient to show that (\ref{strong}) holds  for  every $\xi\in C^\infty(\bar\Omega_T).$ Let $\xi\in C^\infty(\bar\Omega_T).$ By (\ref{pro-1}), $u=D\psi \cdot \frac{x}{|x|^n}$, and hence
\[
\int_{\Omega_T} u\xi_t\,dxdt=\int_{\Omega_T} D\psi \cdot \frac{x}{|x|^n} \xi_t\,dxdt=\lim_{\epsilon\to 0^+} \int_{\Omega_T^\epsilon}D\psi \cdot \frac{x}{|x|^n} \xi_t\,dxdt,
\]
where $\Omega_T^\epsilon=\Omega^\epsilon  \times (0,T)$ with $\Omega^\epsilon= \{\epsilon<|x|<R\}.$ For  all sufficiently small  $\epsilon>0$, by the Divergence Theorem,
\[
\int_{\Omega_T^\epsilon}D\psi \cdot \frac{x}{|x|^n} \xi_t \,dxdt =\int_0^T\int_{\partial \Omega^\epsilon} \psi \xi_t \frac{x}{|x|^{n}}\cdot \mathbf n \,dSdt- \int_{\Omega_T^\epsilon}\psi \mbox{div}\left(\frac{x}{|x|^n} \xi_t\right)  dxdt
\]
\[
=
\frac{1}{R^{n-1}}\int_0^T\int_{|x|=R} \psi\xi_t\,dSdt -
\frac{1}{\epsilon^{n-1}}\int_0^T\int_{|x|=\epsilon} \psi\xi_t\,dSdt
- \int_{\Omega_T^\epsilon}\psi \mbox{div}\left(\frac{x}{|x|^n} \xi_t\right)  dxdt
\]
\[
=:A-B_\epsilon-C_\epsilon,
\]
where $\mathbf n$ is outward unit normal on $\partial\Omega^\epsilon.$
Since $\psi$ is continuous on $\bar \Omega_T$ and $\psi(0,t)=\vp(0,t)=0$ for all $t\in [0,T]$, it is easily seen that
\[
B_\epsilon\to 0\quad \mbox{as $\epsilon\to 0^+.$}
\]
 By (\ref{pro-2}), $\psi(x,t)=C$ for all $(x,t)\in\partial\Omega\times[0,T]$, where $C=\int_0^R w^{n-1}v_0(w)\,dw$ is a constant; hence
\[
A=\frac{1}{R^{n-1}}\int_0^T\int_{|x|=R} (\psi \xi)_t(x,t)\,dSdt=\frac{1}{R^{n-1}}\left[\int_{|x|=R} \psi \xi  dS\right]_0^T.
\]
For the term $C_\epsilon$, using $\mbox{div} \left(\frac{x}{|x|^n}\right)=0$, we have, from integration by parts on $t$,
\[
C_\epsilon = \int_{\Omega_T^\epsilon}\psi  \frac{x}{|x|^n} \cdot D\xi_t \,dxdt=\left[\int_{\Omega^\epsilon}   \frac{x}{|x|^n}  \psi \cdot D\xi \,dx\right]_0^T  -\int_{\Omega_T^\epsilon}\psi_t  \frac{x}{|x|^n} \cdot D\xi \,dxdt
\]
\[=:D_\epsilon -E_\epsilon.
\]
Using the Divergence Theorem and  $\mbox{div} \left(\frac{x}{|x|^n}\right)=0$ again, we have
\[
D_\epsilon=\left[\int_{\Omega^\epsilon}   \frac{x}{|x|^n}  \psi \cdot D\xi \,dx\right]_0^T=\left[\int_{\partial \Omega^\epsilon}  \psi\xi  \frac{x}{|x|^{n}}\cdot \mathbf n \,dS\right]_0^T-\left[\int_{\Omega^\epsilon}  D\psi \cdot \frac{x}{|x|^n}    \xi \,dx\right]_0^T
\]
\[
=\frac{1}{R^{n-1}}\left[\int_{|x|=R} \psi \xi  dS\right]_0^T-\frac{1}{\epsilon^{n-1}}\left[\int_{|x|=\epsilon} \psi \xi  dS\right]_0^T-\left[\int_{\Omega^\epsilon}  u\xi \,dx\right]_0^T
\]
\[
=A-\frac{1}{\epsilon^{n-1}}\left[\int_{|x|=\epsilon} \psi \xi  dS\right]_0^T-\left[\int_{\Omega^\epsilon}  u\xi \,dx\right]_0^T=:A-F_\epsilon -G_\epsilon,
\]
where
\[
\lim_{\epsilon\to 0^+} F_\epsilon = 0 \,\,\,\mbox{ since $\psi(0,t)=0\,\,\forall t\in[0,T]$},\quad \lim_{\epsilon\to 0^+}  G_\epsilon = \left[\int_{\Omega}  u\xi \,dx\right]_0^T.
\]
Finally, using the equation $\psi_t \frac{x}{|x|^n}= a(|Du|^2)Du$ on $\Omega_T$ with $x\neq 0$, we have
\[
E_\epsilon=\int_{\Omega_T^\epsilon}a(|Du|^2)Du \cdot D\xi \,dxdt \to \int_{\Omega_T} a(|Du|^2)Du \cdot D\xi \,dxdt \quad \mbox{as $\epsilon\to 0^+$.}
\]
Therefore
\[
\int_{\Omega_T} u\xi_t\,dxdt =\lim_{\epsilon\to 0^+} (A-B_\epsilon -C_\epsilon)=\lim_{\epsilon\to 0^+} (-B_\epsilon +F_\epsilon+G_\epsilon+E_\epsilon)
\]
\[
=\left[\int_{\Omega}  u\xi \,dx\right]_0^T+\int_{\Omega_T} a(|Du|^2)Du \cdot D\xi \,dxdt.
\]
This is exactly (\ref{strong}), where $u(x,0)=u_0(x)$ in $\Omega$ as shown independently in (c) below. We remark that the fact that $\psi$ is constant on $|x|=R$ plays an important role in the proof. This completes the proof.
\end{proof}

\subsection{Completion of Proof of Theorem \ref{1-8}}

Let us first verify that the radial function $u\in W^{1,\infty}(\Omega_T)$ defined above satisfies all of (a)-(f) in Theorem \ref{1-8}.\\

\underline{\textbf{(a):}} This follows easily from (\ref{strong}).\\

\underline{\textbf{(b):}} From (\ref{3-8}), we have $v\equiv v^*$ on $\overline{J_T\setminus J^*_T}$. So by the definition of $u$,
$$
u\equiv u^*\in C^{3+\alpha,1+\alpha/2}\left(\{\overline{B_{\delta_0}(0)}\cup \overline{(B_{R}(0)\setminus B_{R-\delta_0}(0))}\}\times[0,T]\right).
$$
Observe that
$$
\max_{\{\overline{B_{\delta_0}(0)}\cup \overline{(B_{R}(0)\setminus B_{R-\delta_0}(0))}\}\times[0,T]}|Du|=\max_{\{\overline{B_{\delta_0}(0)}\cup \overline{(B_{R}(0)\setminus B_{R-\delta_0}(0))}\}\times[0,T]}|Du^*|=\max_{\overline{J_T\setminus J^*_T}}|v_s^*|\leq\lambda^-
$$
by (\ref{3-5}). Since $a\equiv a^*$ on $[0,(\lambda^-)^2]$ and $u^*$ solves (\ref{2-2-2}), it follows that $u$ satisfies (b). At the end of this proof, we will check that $\mathcal{C}_\td$ has infinitely many elements $\Phi=(v,\vp)$. The first component $v$ in every $\Phi\in\mathcal{C}_\td$ is then extended to be the common $v^*$ on $\overline{J_T\setminus J^*_T}$, so that each corresponding $u$ satisfies (b) with the same $\delta_0>0$.\\

\underline{\textbf{(c):}} By (\ref{3-3}) and (\ref{3-8}), we have
$$
v(s,0)=v^*(s,0)=v_0(s)\,\,\textrm{ for every } s\in\bar{J}.
$$
Thus from the definitions of $u$ and $v_0$,
$$
u(x,0)=v(|x|,0)=v_0(|x|)=u_0(x)\,\,\textrm{ for every } x\in\bar{\Omega}.
$$

\underline{\textbf{(d):}} This follows immediately from the observation in (b).\\

\underline{\textbf{(e):}} Assume $M\ge 1$; then $\lambda=M+\epsilon.$ Let $(s,t)\in J_T$ be any point such that
$$
\td\Phi(s,t)\in K^{n-1}_{\lambda,l_0}(s,v(s,t))\,\,\textrm{ and }\,\,v_s(s,t)\,\textrm{ exists in }\RR.
$$
Then for every $x\in\Omega$ with $|x|=s$, $Du(x,t)$ exists in $\RR^n$,
$$
|Du(x,t)|=|v_s(s,t)|
$$
by the radial symmetry of $u$, and
$
|v_s(s,t)|\leq\lambda=M+\epsilon
$ by (\ref{3-10}).
Note also that these hold for a.e. $(s,t)\in J_T$, so that
$$
||Du||_{L^\infty(\Omega_T;\RR^n)}=||v_s||_{L^\infty(J_T)}\leq M+\epsilon.
$$

\underline{\textbf{(f):}} This follows easily by taking $\xi\equiv 1$ in (\ref{strong}), which remains valid even when $\Omega_T$ and $T$ are replaced by $\Omega_t$ and $t$ with $0<t\leq T$, respectively.
 \\

 Finally, it remains to check that $\mathcal{C}_\td$ is an infinite set. Suppose on the contrary that $\mathcal{C}_\td$ is finite. Since $\mathcal{C}_\td$ and $\PP^{n-1}_{\lambda,l_0}$ are dense in $X$, we then have $\mathcal{C}_\td=X=\PP^{n-1}_{\lambda,l_0}$. So $\Phi^* \in\PP^{n-1}_{\lambda,l_0}=\mathcal{C}_\td$. By the above, $\Phi^*$ satisfies (\ref{3-10}), that is,
$$
\td\Phi^*(s,t)\in K^{n-1}_{\lambda,l_0}(s,v^*(s,t))\quad\textrm{for a.e.}\,\,(s,t)\in J_T,
$$
and so
$$
(v_s^*(s,t),s^{n-1}\sigma^*(v_s^*(s,t)))=(v_s^*(s,t),\vp_t^*(s,t))\in\tilde{K}^{n-1}_\lambda(s)\quad\textrm{for a.e.}\,\,(s,t)\in J^*_T.
$$
This is equivalent to saying that
$$
(v_s^*(s,t),\sigma^*(v_s^*(s,t)))\in\tilde{K}_\lambda\quad\textrm{for a.e.}\,\,(s,t)\in J^*_T.
$$
By the definition of the set $\tilde{K}_\lambda$, we have
\begin{equation}\label{cont1}
\sigma^*(v_s^*(s,t))=\sigma(v_s^*(s,t))\quad\textrm{for a.e.}\,\,(s,t)\in J^*_T.
\end{equation}
On the other hand, it follows from (\ref{cont2}) and (\ref{3-5}) with $\lambda^-<M$ that choosing a $\delta>0$ so small that
$$
\sigma^*(p)\neq\sigma(p)\quad \forall p\in[-M,-M+\delta]\cup[M-\delta,M],
$$
we have
$$
v^*_s\in[-M,-M+\delta]\cup[M-\delta,M]
$$
on some set $W=W(\delta)\subset J^*_T$ of positive measure. Thus for each $(s,t)\in W$,
$$
\sigma^*(v_s^*(s,t))\neq\sigma(v_s^*(s,t)),
$$
and this is a contradiction to (\ref{cont1}). Therefore, $\mathcal{C}_\td$ is an infinite set.

The theorem is now proved.

\begin{proof}[Remark] Assume $\max_{\bar\Omega}|Du_0|=M<1.$ We select a different $\lambda >1$ such that $\lambda^-=M$ and then select $M'\in (M,\lambda).$ With this choice of $(M',\lambda)$ in place of $(M,\lambda)$ in Lemma \ref{2-3-2}, we construct a function $\sigma^*(p)$. Define $a^*(p):=\sigma^*(\sqrt p)/\sqrt{p}$ for each $p>0$. Then $a^*(p)= a(p)$ for every $p\in(0,M^2]$ and  the functions $a^*$ and $\sigma^*$ satisfy the hypotheses in Theorem \ref{2-2-1}. Therefore, for the given initial condition $u_0$, the problem (\ref{2-2-2}) has a unique radial solution $w^*\in C^{3+\alpha,1+\alpha/2}(\bar\Omega_T).$    Then $w^*$ is also a classical solution to  problem (\ref{1-1}). However, Theorem \ref{1-8} asserts  that, even in this case,  the problem (\ref{1-1}) still has infinitely many weak solutions.
\end{proof}

\section{Auxiliary functions}\label{s4}

In this section, we construct some auxiliary functions that are needed to prove the density theorem, Theorem \ref{2-3-1}.

\subsection{Construction lemma}
We begin  with the following useful result.

\begin{lem}[Construction Lemma]\label{4-3}
Let $a>0$, $b>0$, $L>0$, $s_0>0$, and let $m\geq 0$ be an integer. Let $s_1,\,s_2\in C^1(0,L)$ be two functions satisfying
\begin{equation}\label{4-4}
\left\{ \begin{array}{l}
  0<s_1(t)<s_0<s_2(t),\\
  \frac{s_1(t)+s_2(t)}{2}=s_0,
\end{array} \right.\quad \forall\, t\in (0,L).
\end{equation}
Let $D\subset\RR^2$ be the bounded open set, defined by
$$
D:=\{(s,t)\in\RR^2:0<t<L, s_1(t)<s<s_2(t)\}.
$$
For each $(s,t)\in D$, define
\begin{eqnarray*}
F(s,t)&:=&\int^{\frac{as_1(t)+bs}{a+b}}_{s_1(t)}\tau^m[-a(\tau-s_1(t))]d\tau+\int^{\frac{as_2(t)+bs}{a+b}}_{\frac{as_1(t)+bs}{a+b}}\tau^mb(\tau-s)d\tau \\
&&+\int^{s_2(t)}_{\frac{as_2(t)+bs}{a+b}}\tau^m[-a(\tau-s_2(t))]d\tau.
\end{eqnarray*}
Then we have the following:
\begin{itemize}
\item[(a)] $F\in C^1(D)$,
\item[(b)] there exists a unique function $\tilde{s}\in C^1(0,L)$ such that
$$
s_1(t)<\tilde{s}(t)<s_2(t),\,\,F(\tilde{s}(t),t)=0\quad \forall\,  t\in(0,L),
$$
\item[(c)] $|\tilde{s}'(t)|\leq\left[1+\left(\frac{s_2(t)}{s_1(t)}\right)^m\right]|s_1'(t)|$ for all $t\in (0,L)$,
\item[(d)] if $s_1\in C^1([0,L])$, $s_1(0)>0$, and $s_1(L)>0$, then $\tilde{s}\in C^1([0,L])$.
\end{itemize}
\end{lem}

\begin{proof}
\underline{\textbf{(a):}} Elementary computation shows that for each $(s,t)\in D$,
\begin{equation}\label{4-5}
         F(s,t)= c_m \left [ \frac{(as_1(t)+bs)^{m+2}}{(a+b)^{m+1}}-as_1(t)^{m+2}
         -\frac{(as_2(t)+bs)^{m+2}}{(a+b)^{m+1}} +as_2(t)^{m+2} \right],
\end{equation}
where $c_m= \frac{1}{m+1}-\frac{1}{m+2}>0$.  Since $s_1,s_2\in C^1(0,L)$, it follows immediately from (\ref{4-5}) that $F\in C^1(D).$

\underline{\textbf{(b):}} For each $t\in(0,L)$, using (\ref{4-5}), it can be checked (mainly from the convexity of function $s^{m+2}$ on $s>0$) that $F(s_1(t),t)>0$ and $F(s_2(t),t)<0.$ Moreover, on $D$,
$$
\partial_s F(s,t)=\frac{-b c_m}{(a+b)^{m+1}} [(as_2(t)+bs)^{m+1}-(as_1(t)+bs)^{m+1}]<0,
$$
since $s_2(t)>s_1(t)>0$. In particular, $\partial_s F(s,t)\neq 0$ for every $(s,t)\in D$. Therefore, by the Intermediate Value Theorem, for each $t \in (0,L)$, there exists a unique $\tilde s(t)\in (s_1(t),s_2(t))$ such that
\[
F(\tilde s(t),t)=0.
\]
Furthermore, by the Implicit Function Theorem, it follows that $\tilde{s}\in C^1 (0,L),$
and so (b) is proved.

\underline{\textbf{(c):}} Clearly, by (\ref{4-5}), $\tilde s(t)$ satisfies the equation
\begin{equation}\label{4-6}
\begin{array}{c}
  0=(as_1(t)+b\tilde{s}(t))^{m+2}-a(a+b)^{m+1}s_1(t)^{m+2} \\
  -(as_2(t)+b\tilde{s}(t))^{m+2}+a(a+b)^{m+1}s_2(t)^{m+2}
\end{array}
\end{equation}
for each $t\in(0,L)$. Taking derivatives on both sides in (\ref{4-6}) with respect to $t$, we obtain
\begin{eqnarray*}
\tilde{s}'(t) &=& \frac{a}{b}\Big\{\frac{s_1'(t)[(as_1(t)+b\tilde{s}(t))^{m+1}-(as_1(t)+bs_1(t))^{m+1}]}{(as_2(t)+b\tilde{s}(t))^{m+1}-(as_1(t)+b\tilde{s}(t))^{m+1}}\\
&& +\frac{s_2'(t)[(as_2(t)+bs_2(t))^{m+1}-(as_2(t)+b\tilde{s}(t))^{m+1}]}{(as_2(t)+b\tilde{s}(t))^{m+1}-(as_1(t)+b\tilde{s}(t))^{m+1}}\Big\}
\end{eqnarray*}
for each $t\in(0,L)$. Applying the Mean Value Theorem, we have
\begin{eqnarray*}
\tilde{s}'(t) &=& s_1'(t)\left(\frac{as_1(t)+b\bar{s}_1(t)}{a\bar{s}_3(t)+b\tilde{s}(t)}\right)^m\frac{\tilde{s}(t)-s_1(t)}{s_2(t)-s_1(t)}\\
&& + s_2'(t)\left(\frac{as_2(t)+b\bar{s}_2(t)}{a\bar{s}_3(t)+b\tilde{s}(t)}\right)^m\frac{s_2(t)-\tilde{s}(t)}{s_2(t)-s_1(t)}
\end{eqnarray*}
for some $\bar{s}_1(t),\,\bar{s}_2(t),\,\bar{s}_3(t)\in\RR$ with
$$
s_1(t)<\bar{s}_1(t)<\tilde{s}(t) <\bar{s}_2(t)<s_2(t),\quad s_1(t)<\bar{s}_3(t)<s_2(t).
$$
So
$$
|\tilde{s}'(t)|\leq|s_1'(t)|+|s_1'(t)|\left(\frac{s_2(t)}{s_1(t)}\right)^m=\left[1+\left(\frac{s_2(t)}{s_1(t)}\right)^m\right]|s_1'(t)|,
$$
since $s_2'(t)=-s_1'(t)$ by (\ref{4-4}). Thus (c) is proved.

\underline{\textbf{(d):}} Finally to prove (d), assume
\begin{equation}\label{bdry1}
s_1\in C^1([0,L))\,\,\textrm{ and }\,\,s_1(0)>0.
\end{equation}
We will show that $\tilde{s}\in C^1([0,L))$. If $s_1\in C^1((0,L])$ and $s_1(L)>0$, we can prove that $\tilde{s}\in C^1((0,L])$ exactly in the same way. Returning to the assumption (\ref{bdry1}), note $0<s_1(0)\leq s_0$ by (\ref{4-4}). If $0<s_1(0)<s_0$, then we can extend $s_1$ and $s_2$ from $[0,L)$ to $(-\delta,L)$ for some $\delta>0$ in a way that $s_1,s_2\in C^1((-\delta,L))$ satisfy (\ref{4-4}), and we can apply the previous argument to show $\tilde{s}\in C^1((-\delta,L))$. So let us assume $s_1(0)=s_0$; then, by (\ref{4-4}), $s_2(0)=s_0.$  From (\ref{4-6}), we have
$$\lim_{t\RA 0^+}\tilde{s}(t)=s_1(0)=s_2(0)=s_0.$$
We claim that
\begin{equation}\label{claim}
\lim_{t\RA 0^+}\tilde{s}'(t)=0,
\end{equation}
 and so $\tilde{s}\in C^1([0,L))$. To prove this claim, we   rewrite (\ref{4-6}) as
\begin{equation}\label{4-7}
f(as_1(t)+b\tilde{s}(t),as_2(t)+b\tilde{s}(t))=(a+b)^{m+1}f(s_1(t),s_2(t)) \quad \forall\,t\in (0,L),
\end{equation}
 where $f(s_1,s_2)$ is the polynomial in $s_1$ and  $s_2$ determined through
\[
s_1^{m+2}-s_2^{m+2} =(s_1-s_2)f(s_1,s_2)\quad \forall\,\, (s_1, s_2)\in \RR^2.
\]
 Note that $f(s_1,s_2)$ is symmetric in $(s_1, s_2)$ and $\partial_{1}f(s_1,s_2)>0$ for all $s_1>0$ and $s_2>0$.
To prove (\ref{claim}), note that, by (c), $\tilde{s}'(t)$ is bounded on $(0,L/2)$,  and  it suffices to show that if  $\beta:=\lim_{k \RA\infty}\tilde{s}'(t_{k})$ exists along a sequence $t_k\to 0^+$, then $\beta=0.$ Let $\alpha=s_1'(0^+);$ then $s_2'(0^+)=-\alpha.$  Taking derivatives on both sides in (\ref{4-7}) with respect to $t$, we have
\[
\partial_1 f(as_2(t)+b\tilde{s}(t),as_1(t)+b\tilde{s}(t))\cdot(as_2'(t)+b\tilde{s}'(t))
\]
\[
+\partial_2 f(as_2(t)+b\tilde{s}(t),as_1(t)+b\tilde{s}(t))\cdot(as_1'(t)+b\tilde{s}'(t))
\]
\[
=(a+b)^{m+1}[\partial_1 f(s_2(t),s_1(t))s_2'(t)+\partial_2 f(s_2(t),s_1(t))s_1'(t)]
\]
for each $t\in (0,L)$. Letting $t=t_k\to 0^+$, we have
$$
\partial_1 f(as_0+bs_0,as_0+bs_0)\cdot(-a\alpha+b\beta)
+\partial_2 f(as_0+bs_0,as_0+bs_0)\cdot(a\alpha+b\beta)
$$
$$
=(a+b)^{m+1}[\partial_1 f(s_0,s_0)(-\alpha)+\partial_2 f(s_0,s_0)\alpha]=0,
$$
by the symmetry of $f$. This yields that $ 2\beta b\partial_1 f(as_0+bs_0,as_0+bs_0) =0$, so that $\beta=0$, as wished. Hence (\ref{claim}) follows, and (d) is proved.
\end{proof}

\subsection{Construction of auxiliary functions}

We are now ready to construct auxiliary functions that will be used as local gradient modifiers in the proof of Theorem \ref{2-3-1}.
Towards this goal, let $a>0$, $b>0$, $L>0$, $s^2_0>s^1_0>0$, $s_0:=\frac{s^2_0+s^1_0}{2}$, and let $m\geq 0$ be an integer. Define
$$
s_1(t):=\frac{s^2_0-s^1_0}{2L} t+s^1_0\,\,\textrm{ and }\,\,s_2(t):=\frac{s^1_0-s^2_0}{2L} t+s^2_0
$$
for each $t\in[0,L]$. (See Figure \ref{fig3} with $t_0=0$.)
\begin{figure}[ht]
\begin{center}
\includegraphics[scale=0.55]{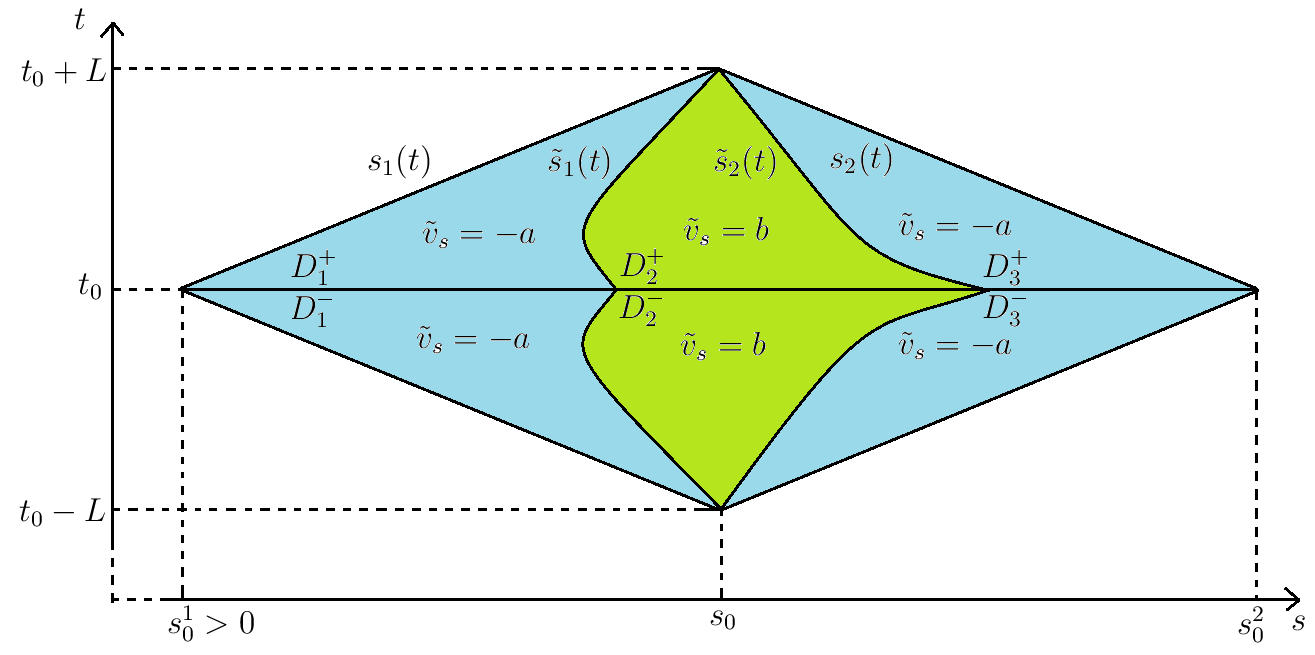}
\caption{The $s$-derivatives of $\tilde{v}$ in Lemma \ref{4-16} on the six regions separated by nonlinear piecewise $C^1$ curves $\tilde{s}_1(t)$ and $\tilde{s}_2(t)$}
\label{fig3}
\end{center}
\end{figure}

Let $D^+\subset\RR^2$ be the bounded open set, given by
$$
D^+:=\{(s,t)\in\RR^2:0<t<L,\,s_1(t)<s<s_2(t)\}.
$$
For each $(s,t)\in D^+$, define $F(s,t)$ as in Lemma \ref{4-3}, so that there exists  a unique $\tilde{s}\in C^1([0,L])$ such that
\begin{equation}\label{4-8}
\left\{ \begin{array}{ll}
  s_1(t)<\tilde{s}(t)<s_2(t),\,F(\tilde{s}(t),t)=0 & \forall t\in[0,L), \\
  |\tilde{s}'(t)|\leq\left[1+\left(\frac{s_2(t)}{s_1(t)}\right)^m\right]|s_1'(t)| & \forall t\in[0,L].
\end{array} \right.
\end{equation}
Let $\tilde{s}_1,\tilde{s}_2:[0,L]\RA\RR$ be given by
$$
\tilde{s}_i(t):=\frac{as_i(t)+b\tilde{s}(t)}{a+b}\,\,\,\,\,\forall t\in[0,L],\,\forall i\in\{1,2\},
$$
so that $\tilde{s}_1,\tilde{s}_2\in C^1([0,L])$ and that by (\ref{4-8}),
$$
s_1(t)<\tilde{s}_1(t)<\tilde{s}(t)<\tilde{s}_2(t)<s_2(t)\,\,\,\,\,\forall t\in[0,L).
$$
Let $D^+_1,D^+_2,D^+_3\subset\RR^2$ be the bounded open sets, defined by
\[
D^+_1:=\{(s,t)\in\RR^2:0<t<L,\,s_1(t)<s<\tilde{s}_1(t)\},
\]
\[
D^+_2:=\{(s,t)\in\RR^2:0<t<L,\,\tilde{s}_1(t)<s<\tilde{s}_2(t)\},
\]
\[
D^+_3:=\{(s,t)\in\RR^2:0<t<L,\,\tilde{s}_2(t)<s<s_2(t)\},
\]
so that these are disjoint open subsets of $D^+$ with
$$
\Big|D^+\setminus\cup_{i=1}^3 D^+_i\Big|=0.\quad\textrm{(See Figure \ref{fig3} with $t_0=0$.)}
$$
Let $\tilde{v}:\bar{D}^+\RA\RR$ be the function, defined by
\begin{equation*}
\tilde{v}(s,t):=\left\{ \begin{array}{ll}
  -a(s-s_1(t)) & \forall (s,t)\in\bar{D}^+_1 \\
  b(s-\tilde{s}(t)) & \forall (s,t)\in\bar{D}^+_2 \\
  -a(s-s_2(t)) & \forall (s,t)\in\bar{D}^+_3.
\end{array} \right.
\end{equation*}
It is easily checked that $\tilde{v}:\bar{D}^+\RA\RR$ is well-defined and that $\tilde{v}\in W^{1,\infty}(D^+)$. It also follows from Lemma \ref{4-3} that $\tilde{v}\in C^1(\bar{D}^+_i)$ for $i=1,2,3$. If $0\leq t\leq L$, then
$$
\tilde{v}(s_1(t),t)=\tilde{v}(s_2(t),t)=0.
$$
Let $t\in[0,L]$. Then
$$
\max_{s\in[s_1(t),s_2(t)]}|\tilde{v}(s,t)|=\max\{|-a(\tilde{s}_1(t)-s_1(t))|,|-a(\tilde{s}_2(t)-s_2(t))|\}
$$
$$
=\max\left\{\left|-a\left(\frac{as_1(t)+b\tilde{s}(t)}{a+b}-s_1(t)\right)\right|,\left|-a\left(\frac{as_2(t)+b\tilde{s}(t)}{a+b}-s_2(t)\right)\right|\right\}
$$
$$
=\max\left\{\frac{ab}{a+b}(\tilde{s}(t)-s_1(t)),\frac{ab}{a+b}(s_2(t)-\tilde{s}(t))\right\}\leq\frac{ab}{a+b}(s^2_0-s^1_0).
$$
Hence
$$
\max_{\bar{D}^+}|\tilde{v}|\leq\frac{ab}{a+b}(s^2_0-s^1_0)\leq\frac{a+b}{4}(s^2_0-s^1_0).
$$
Define
$$
D^-:=\{(s,t)\in\RR^2:(s,-t)\in D^+\},
$$
$$
D^-_i:=\{(s,t)\in\RR^2:(s,-t)\in D^+_i\}\,\,\,\,\forall i\in\{1,2,3\},
$$
$$
D:=\mathrm{int}(\overline{D^+\cup D^-}).
$$
We do the even extensions for $s_1,\tilde{s}_1,\tilde{s},\tilde{s}_2,s_2:[-L,L]\RA\RR$ and for $\tilde{v}:\bar{D}\RA\RR$ along the $t$-axis, so that we have from the above observations that
\begin{equation}\label{4-9}
\left\{ \begin{array}{l}
  \tilde{v}\in W^{1,\infty}_0(D), \\
  \tilde{v}\in C^1(\bar{D}^\pm_i)\,\,\,\forall i\in\{1,2,3\}, \\
  \max_{\bar{D}}|\tilde{v}|\leq\frac{a+b}{4}(s^2_0-s^1_0).
\end{array} \right.
\end{equation}
It follows from (\ref{4-8}) that for each $t\in[0,L]$,
\begin{equation}\label{4-10}
\int_{s_1(t)}^{s_2(t)}\tau^m\tilde{v}(\tau,t)d\tau=F(\tilde{s}(t),t)=0,
\end{equation}
and this equality is valid for all $t\in[-L,L]$ by the definition $\tilde{v}$.
Note also that
\begin{equation}\label{4-11}
\td\tilde{v}(s,t)=\left\{ \begin{array}{ll}
  (-a,as_1'(t)) & \textrm{ if }(s,t)\in D^+_1, \\
  (b,-b\tilde{s}'(t)) & \textrm{ if }(s,t)\in D^+_2, \\
  (-a,as_2'(t)) & \textrm{ if }(s,t)\in D^+_3, \\
  (-a,-as_1'(-t)) & \textrm{ if }(s,t)\in D^-_1, \\
  (b,b\tilde{s}'(-t)) & \textrm{ if }(s,t)\in D^-_2, \\
  (-a,-as_2'(-t)) & \textrm{ if }(s,t)\in D^-_3.
\end{array} \right.
\end{equation}
Also the second of (\ref{4-8}) implies that
$$
|\tilde{s}'(t)|\leq\left[1+\left(\frac{s^2_0}{s^1_0}\right)^m\right]\frac{s^2_0-s^1_0}{2L}\,\,\,\,\,\forall t\in[-L,L].
$$
Combining this with (\ref{4-11}), we have
\begin{equation}\label{4-12}
|\partial_t\tilde{v}(s,t)|\leq\max\{a,b\}\left[1+\left(\frac{s^2_0}{s^1_0}\right)^m\right]\frac{s^2_0-s^1_0}{2L}\,\,\,\,\,\forall (s,t)\in\bigcup_{i=1}^3 (D^+_i\cup D^-_i).
\end{equation}
Using the third of (\ref{4-9}), we obtain
\begin{equation}\label{4-13}
\max_{(s,t)\in\bar{D}}\left|\int^s_{s_1(t)}\tau^m\tilde{v}(\tau,t)d\tau\right|\leq\frac{a+b}{4}(s^2_0)^m(s^2_0-s^1_0)^2.
\end{equation}
One can also easily check that
\begin{equation}\label{4-14}
\frac{\partial}{\partial t}\left(\int^s_{s_1(t)}\tau^m\tilde{v}(\tau,t)d\tau\right)=\int^s_{s_1(t)}\tau^m\partial_t\tilde{v}(\tau,t)d\tau\,\,\,\,\,\forall (s,t)\in\bigcup_{i=1}^3 (D^+_i\cup D^-_i).
\end{equation}

Fix any $t_0\in\RR$. We now translate everything constructed above along the $t$-axis by $t_0$. So we define
\begin{equation}\label{4-15}
\begin{array}{rcl}
  D(s^1_0,s^2_0,t_0,L) & := & \{(s,t)\in\RR^2:(s,t-t_0)\in D\}, \\
  D^\pm_i(s^1_0,s^2_0,t_0,L) & := & \{(s,t)\in\RR^2:(s,t-t_0)\in D^\pm_i\}\,\,\,\,\,\forall i\in\{1,2,3\}, \\
  s_j(s^1_0,s^2_0,t_0,L;t) & := & s_j(t-t_0)\,\,\,\,\,\forall t\in[t_0-L,t_0+L],\,\,\forall j\in\{1,2\}, \\
  \tilde{v}(-a,b,s^1_0,s^2_0,t_0,L;s,t) & := & \tilde{v}(s,t-t_0)\,\,\,\,\,\forall (s,t)\in\overline{D(s^1_0,s^2_0,t_0,L)}.
\end{array}
\end{equation}
Here is the right spot of mentioning a rather delicate feature of our construction. We should prohibit the auxiliary function $\tilde{v}$ in (\ref{4-9}) from being translated in the $s$-axis as any $s$-translation will destroy the key properties to act as auxiliary functions for local gluing in the proof of the density theorem, Theorem \ref{2-3-1}. Accordingly, we construct $\tilde{v}$ on the positive $s$-axis from the start and allow translation in the $t$-axis only as in (\ref{4-15}).

As a conclusion of this section, we suppress the letters $-a,\,b,\,s^1_0,\,s^2_0,\,t_0,\,L$ in (\ref{4-15}) for a notational simplicity and summarize the properties of $\tilde{v}$ inherited from (\ref{4-9}), (\ref{4-10}), (\ref{4-11}), (\ref{4-12}), (\ref{4-13}), and (\ref{4-14}) as follows. (See Figure \ref{fig3}.)
\begin{lem}\label{4-16}
The function $\tilde{v}:\bar{D}\RA\RR$ constructed in (\ref{4-15}) satisfies the following:
\begin{itemize}
\item[(a)] $\tilde{v}\in W^{1,\infty}_0(D)$,
\item[(b)] $\tilde{v}\in C^1(\bar{D}^\pm_i)\,\,\,\,\,\forall i=1,2,3$,
\item[(c)] $\partial_s\tilde{v}(s,t)=\left\{\begin{array}{cl}
                                       -a & \forall (s,t)\in D^+_1\cup D^-_1\cup D^+_3\cup D^-_3 \\
                                       b & \forall (s,t)\in D^+_2\cup D^-_2,
                                     \end{array}\right.$
\item[(d)] $|\partial_t\tilde{v}(s,t)|\leq\max\{a,b\}\left[1+\left(\frac{s^2_0}{s^1_0}\right)^m\right]\frac{s^2_0-s^1_0}{2L}\,\,\,\,\,\forall (s,t)\in\bigcup_{i=1}^3(D^+_i\cup D^-_i),$
\item[(e)] $\frac{\partial}{\partial t}\left(\int^s_{s_1(t)}\tau^m\tilde{v}(\tau,t)d\tau\right)=\int^s_{s_1(t)}\tau^m\partial_t\tilde{v}(\tau,t)d\tau\,\,\,\,\,\forall (s,t)\in\bigcup_{i=1}^3 (D^+_i\cup D^-_i)$,
\item[(f)] $\int_{s_1(t)}^{s_2(t)}\tau^m\tilde{v}(\tau,t)d\tau=0\,\,\,\,\,\forall t\in[t_0-L,t_0+L]$,
\item[(g)] $\max_{\bar{D}}|\tilde{v}|\leq\frac{a+b}{4}(s^2_0-s^1_0)$,
\item[(h)] $\max_{(s,t)\in\bar{D}}\left|\int^s_{s_1(t)}\tau^m\tilde{v}(\tau,t)d\tau\right|\leq\frac{a+b}{4}(s^2_0)^m(s^2_0-s^1_0)^2.$
\end{itemize}
\end{lem}


\section{Proof of Theorem \ref{2-3-1}}\label{s5}
In this long and final section, we present the proof of Theorem \ref{2-3-1}; we divide the proof into several parts.

Fix an $\epsilon>0$. Let $\Phi=(v,\vp)\in \mathcal{P}^m_{\lambda,l_0}$, i.e.,
\begin{equation}\label{5-1}
\left\{ \begin{array}{l}
          \Phi\in\Phi^*+W^{1,\infty}_0(J^*_T;\RR^2), \\
          \textrm{$\Phi$ is piecewise $C^1$ in $J^*_T$, and} \\
          \td\Phi(s,t)\in K^m_{\lambda,l_0}(s,v(s,t))\cup U^m_{\lambda,l_0}(s,v(s,t))\,\,\,\textrm{ for a.e. }(s,t)\in J^*_T.
        \end{array} \right.
\end{equation}
Let $0<\eta<1$. Our goal is to construct a function $\Phi_\eta\in \mathcal{P}^m_{\lambda,l_0,\epsilon}$ such that $||\Phi-\Phi_\eta||_{L^{\infty}(J^*_T;\RR^2)}\leq\eta$, i.e., to construct a function
  $\Phi_\eta=(v_\eta,\vp_\eta)\in \Phi^*+W^{1,\infty}_0(J^*_T;\RR^2)$ satisfying that
\begin{equation}\label{5-2}
\left\{ \begin{array}{l}
          \textrm{$\Phi_\eta$ is piecewise $C^1$ in $J^*_T$,} \\
          \td\Phi_\eta(s,t)\in K^m_{\lambda,l_0}(s,v_\eta(s,t))\cup U^m_{\lambda,l_0}(s,v_\eta(s,t))\,\,\,\textrm{ for a.e. }(s,t)\in J^*_T, \\
          \int_{J^*_T}\mathrm{dist}(\td\Phi_\eta(s,t), K^m_{\lambda,l_0}(s,v_\eta(s,t)))dsdt \leq \epsilon|J^*_T|,\\
||\Phi-\Phi_\eta||_{L^{\infty}(J^*_T;\RR^2)}\leq\eta.
        \end{array} \right.
\end{equation}

\subsection{Separation of domain $J^*_T$}
By the second of (\ref{5-1}), there is a sequence $\{G_i\}_{i\in\NN}$ of disjoint open subsets of $J^*_T$ such that
\begin{equation*}
\left\{ \begin{array}{l}
          |J^*_T\setminus \cup_{i=1}^\infty G_i|=0, \\
          \Phi\in C^1(\bar{G}_i;\RR^2)\,\,\,\forall i\in\NN.
        \end{array} \right.
\end{equation*}
Fix an index $i\in\NN$ throughout this section. Since $\partial_s \vp$ and $v$ are continuous on $\bar{G}_i$, it follows from the third inclusion of (\ref{5-1}) that
$$
\partial_s\vp(s,t)=s^m v(s,t),\,\,\textrm{ i.e., }\,\,\td\Phi(s,t)\in W_{s^m v(s,t)}\,\,\,\,\forall (s,t)\in\bar{G}_i.
$$
Applying Lemma \ref{2-3-7}, we have
\begin{eqnarray*}
d_i(s,t) &:=& \mathrm{dist}(\td\Phi(s,t),K^m_{\lambda,l_0}(s,v(s,t))\cup \partial|_{W_{s^m v(s,t)}}U^m_{\lambda,l_0}(s,v(s,t)))\\
&=& \mathrm{dist}(P_W(\td\Phi(s,t)),K^m_{\lambda,l_0}(s,0)\cup \partial|_{W}U^m_{\lambda,l_0}(s,0))
\end{eqnarray*}
for every $(s,t)\in\bar{G}_i$, and it follows form Lemma \ref{2-3-8} that the mapping $d_i:\bar{G}_i\RA[0,\infty)$ is continuous. Let $0<\delta<1$, and put
$$
K_{i,\delta}:=\{(s,t)\in G_i:d_i(s,t)\leq\delta\}.
$$
Define also
\begin{equation*}
\begin{array}{l}
  K^1_{i,\delta}:=\{(s,t)\in K_{i,\delta}:\td\Phi(s,t)\not\in K^m_{\lambda,l_0}(s,v(s,t))\cup U^m_{\lambda,l_0}(s,v(s,t))\}, \\
  K^2_{i,\delta}:=\{(s,t)\in K_{i,\delta}:\td\Phi(s,t)\in K^m_{\lambda,l_0}(s,v(s,t))\}, \\
  K^3_{i,\delta}:=\{(s,t)\in K_{i,\delta}:\td\Phi(s,t)\in U^m_{\lambda,l_0}(s,v(s,t))\},\\
\end{array}
\end{equation*}
so that $K_{i,\delta}$ is the disjoint union of $K^1_{i,\delta}$, $K^2_{i,\delta}$, and $K^3_{i,\delta}$. Note that $|K^1_{i,\delta}|=0$ by the third of (\ref{5-1}), and that
\begin{equation*}
\begin{array}{c}
  K^3_{i,\delta} \subseteq
  \{(s,t)\in G_i:\mathrm{dist}(\td\Phi(s,t),K^m_{\lambda,l_0}(s,v(s,t)))\leq\delta, \td\Phi(s,t)\in U^m_{\lambda,l_0}(s,v(s,t))\} \\
  \cup \{(s,t)\in G_i:\mathrm{dist}(\td\Phi(s,t),\partial|_{W_{s^m v(s,t)}}U^m_{\lambda,l_0}(s,v(s,t)))\leq\delta, \td\Phi(s,t)\in U^m_{\lambda,l_0}(s,v(s,t))\} \\
  =:K^{3,\alpha}_{i,\delta}\cup K^{3,\beta}_{i,\delta}.
\end{array}
\end{equation*}
Hence
\begin{equation}\label{5-6}
\begin{array}{c}
  \int_{K_{i,\delta}}\mathrm{dist}(\td\Phi(s,t),K^m_{\lambda,l_0}(s,v(s,t))) \leq \int_{K^2_{i,\delta}}\mathrm{dist}(\td\Phi(s,t),K^m_{\lambda,l_0}(s,v(s,t)))  \\
 +\int_{K^{3,\alpha}_{i,\delta}}\mathrm{dist}(\td\Phi(s,t),K^m_{\lambda,l_0}(s,v(s,t)))
  +\int_{K^{3,\beta}_{i,\delta}}\mathrm{dist}(\td\Phi(s,t),K^m_{\lambda,l_0}(s,v(s,t))) \\
  =\int_{K^{3,\alpha}_{i,\delta}}\mathrm{dist}(\td\Phi(s,t),K^m_{\lambda,l_0}(s,v(s,t)))  +\int_{K^{3,\beta}_{i,\delta}}\mathrm{dist}(\td\Phi(s,t),K^m_{\lambda,l_0}(s,v(s,t))) \\
  \leq \delta|K^{3,\alpha}_{i,\delta}|+N_i|K^{3,\beta}_{i,\delta}| \leq \delta|J^*_T|+N_i|K^{3,\beta}_{i,\delta}|,
\end{array}
\end{equation}
where $N_i:=\max_{(s,t)\in\bar{G}_i} \mathrm{dist}(\td\Phi(s,t),K^m_{\lambda,l_0}(s,v(s,t)))$ is independent of $\delta$.
By the definition of $K^{3,\beta}_{i,\delta}$ $(0<\delta<1)$,
$$
K^{3,\beta}_{i,\delta_1}\subset K^{3,\beta}_{i,\delta_2}\,\,\textrm{ whenever }\,\,0<\delta_1<\delta_2<1.
$$
Let us check that
\begin{equation}\label{5-4}
\bigcap_{0<\delta<1}K^{3,\beta}_{i,\delta}=\emptyset.
\end{equation}
Suppose on the contrary that there is a point $(s,t)\in\bigcap_{0<\delta<1}K^{3,\beta}_{i,\delta}.$
Then
$$
\mathrm{dist}(\td\Phi(s,t),\partial|_{W_{s^m v(s,t)}}U^m_{\lambda,l_0}(s,v(s,t)))=0,
$$
and so
$$
\td\Phi(s,t)\in U^m_{\lambda,l_0}(s,v(s,t))\cap \partial|_{W_{s^m v(s,t)}}U^m_{\lambda,l_0}(s,v(s,t))\neq\emptyset.
$$
This is a contradiction to the fact that $U^m_{\lambda,l_0}(s,v(s,t))$ is open in $W_{s^m v(s,t)}$, and so (\ref{5-4}) holds. We thus have
$$
\delta|J^*_T|+N_i|K^{3,\beta}_{i,\delta}|\RA 0\,\,\textrm{ as }\,\,\delta\to 0^+.
$$
Note also that
$$
|\{(s,t)\in G_i:d_i(s,t)=\delta\}|>0
$$
for at most countably many $\delta\in(0,1)$. So it is possible to choose a $\delta_i\in(0,\epsilon/2)$ so that
\begin{equation}\label{5-5}
\left\{\begin{array}{l}
  \delta_i |J^*_T|+N_i|K^{3,\beta}_{i,\delta_i}|\leq\frac{\epsilon}{2^{i+1}}|J^*_T|, \\
  |\{(s,t)\in G_i:d_i(s,t)=\delta_i\}|=0.
\end{array}\right.
\end{equation}
With this choice of $\delta_i$, we define
\begin{eqnarray*}
\hat{K}_i &:=& \{(s,t)\in G_i:d_i(s,t)<\delta_i\},\\
\hat{H}_i &:=& \{(s,t)\in G_i:d_i(s,t)=\delta_i\},\\
\hat{G}_i &:=& \{(s,t)\in G_i:d_i(s,t)>\delta_i\},
\end{eqnarray*}
so that $K_{i,\delta_i}=\hat{K}_i\cup\hat{H}_i$, $|\hat{H}_i|=0$ by (\ref{5-5}), and $\hat{K}_i$ and $\hat{G}_i$ are disjoint open subsets of $G_i$ with $|G_i\setminus(\hat{K}_i\cup\hat{G}_i)|=0$ by the continuity of the mapping $d_i:\bar{G}_i\RA[0,\infty)$. By (\ref{5-6}) and (\ref{5-5}), we have
\begin{equation}\label{5-7}
\begin{array}{c}
  \int_{\hat{K_i}}\mathrm{dist}(\td\Phi(s,t),K^m_{\lambda,l_0}(s,v(s,t)))dsdt \\
  =\int_{K_{i,\delta_i}}\mathrm{dist}(\td\Phi(s,t),K^m_{\lambda,l_0}(s,v(s,t)))dsdt\leq\frac{\epsilon}{2^{i+1}}|J^*_T|.
\end{array}
\end{equation}

Let us take a moment here to explain what we have done so far. We have separated the open set $G_i$ into two disjoint open sets $\hat{K}_i$ and $\hat{G}_i$. On the set $\hat{K}_i$, the value of the integral in question is already ``small" enough to the extent (\ref{5-7}) as we wanted in the fulfillment of the third of (\ref{5-2}). So no modification will be made to $\Phi$ on the set $\hat{K}_i$. But on the set $\hat{G}_i$, the (inhomogeneous) distance from the gradient of $\Phi$ to $K^m_{\lambda,l_0}$ is relatively ``large", and therefore a necessary modification will be made to $\Phi$ by gluing suitable functions constructed in Section \ref{s4}, specifically in Lemma \ref{4-16}, so that the integral can be made ``small" enough. This is what to be accomplished in the following subsections.

\begin{figure}[ht]
\begin{center}
\includegraphics[scale=0.6]{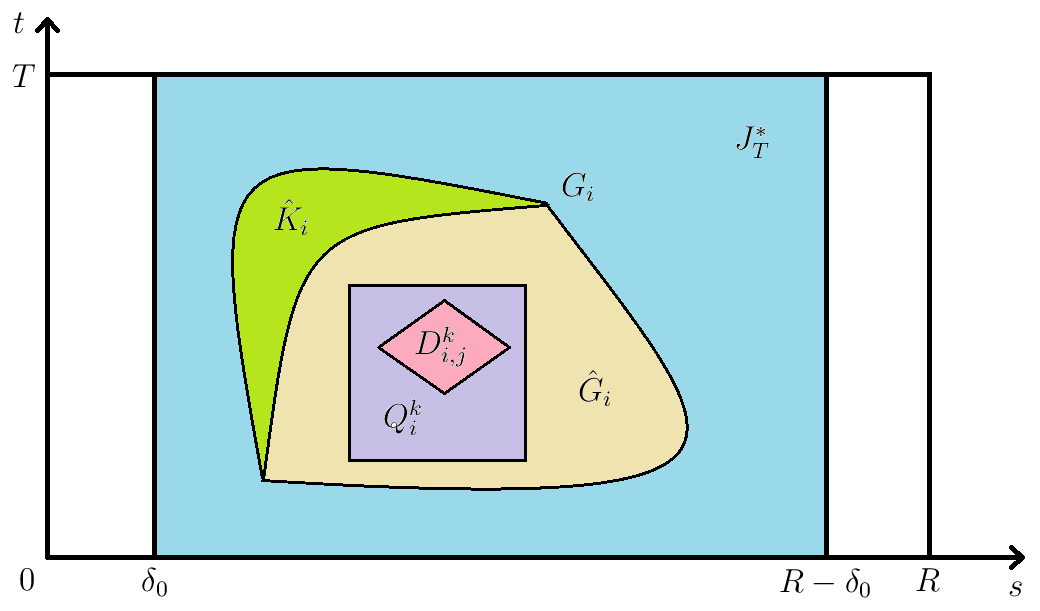}
\end{center}
\caption{Heuristic diagram of subdivisions of $J^*_T$}
\label{fig6}
\end{figure}

\subsection{Properties of the gradient of $\Phi$ in $\hat{G}_i$}
By the uniform continuity of $\td\Phi:\bar{G}_i\RA\RR^{2\times2}$, there exists an $\eta_i=\eta_i(\rho,\delta_i)>0$ such that
\begin{equation}\label{5-8}
(s,t),(s',t')\in\bar{G}_i,\,|(s,t)-(s',t')|\leq\eta_i \Rightarrow |\td\Phi(s,t)-\td\Phi(s',t')|\leq\rho\delta_i,
\end{equation}
where $\rho>0$ is a constant with
\begin{equation}\label{5-9}
\rho<\min\left\{\frac{1}{6},\frac{1}{12R^mM_\sigma}\right\}
\end{equation}
and
$$
M_\sigma:=\sup_{p_1,p_2\in[-2\lambda,2\lambda],p_1\neq p_2}\left|\frac{\sigma(p_1)-\sigma(p_2)}{p_1-p_2}\right|<\infty.
$$
Let us check that for each $(s,t)\in\hat{G}_i$,
\begin{equation}\label{5-12}
\left\{ \begin{array}{l}
          \td\Phi(s,t)\in U^m_{\lambda,l_0}(s,v(s,t)), \\
          \mathrm{dist}((\partial_s v(s,t),\partial_t\vp(s,t)),\partial\tilde{U}^m_\lambda(s))>\delta_i, \\
          |\partial_t v(s,t)|<l_0-\delta_i.
        \end{array} \right.
\end{equation}
To show this, choose any $(s,t)\in\hat{G}_i$. By the third of (\ref{5-1}), we can take a sequence $\{(s_j,t_j)\}_{j\in\NN}$ in $\hat{G}_i$ such that
\begin{equation*}
\left\{ \begin{array}{l}
          \td\Phi(s_j,t_j)\in K^m_{\lambda,l_0}(s_j,v(s_j,t_j))\cup U^m_{\lambda,l_0}(s_j,v(s_j,t_j))\,\,\,\forall j\in\NN, \\
          (s_j,t_j)\RA(s,j)\,\,\,\textrm{ in $\RR^2$ as $j\RA\infty$.}
        \end{array} \right.
\end{equation*}
So for each $j\in\NN$, we have
\begin{equation}\label{5-13}
\left(
  \begin{array}{cc}
    \partial_s v(s_j,t_j) & \partial_t v(s_j,t_j) \\
    \partial_s \vp(s_j,t_j) & \partial_t \vp(s_j,t_j) \\
  \end{array}
\right) = \left(
            \begin{array}{cc}
              p_j & l_j \\
              (s_j)^m v(s_j,t_j) & (s_j)^mq_j \\
            \end{array}
          \right)
\end{equation}
for some $(p_j,q_j)\in\tilde{K}_\lambda\cup\tilde{U}^+_\lambda\cup\tilde{U}^-_\lambda$ and some $l_j\in[-l_0,l_0]$. Passing to a subsequence (we do not relabel),
$$
(p_j,q_j)\RA(p,q)\,\,\textrm{ and }\,\, l_j\RA l\,\,\,\textrm{ as }j\RA\infty,
$$
for some $(p,q)\in\overline{\tilde{K}_\lambda\cup\tilde{U}^+_\lambda\cup\tilde{U}^-_\lambda}$ and some $l\in[-l_0,l_0]$. Letting $j\RA\infty$ on both sides of (\ref{5-13}), we obtain
\begin{equation}\label{5-14}
\left(
  \begin{array}{cc}
    \partial_s v(s,t) & \partial_t v(s,t) \\
    \partial_s \vp(s,t) & \partial_t \vp(s,t) \\
  \end{array}
\right) = \left(
            \begin{array}{cc}
              p & l \\
              s^m v(s,t) & s^mq \\
            \end{array}
          \right).
\end{equation}
Since $(s,t)\in\hat{G}_i$, it follows from the definition of $\hat{G}_i$ that
\begin{eqnarray}\label{5-15}
\delta_i<d_i(s,t) &=& \mathrm{dist}(\td\Phi(s,t),K^m_{\lambda,l_0}(s,v(s,t))\cup \partial|_{W_{s^m v(s,t)}}U^m_{\lambda,l_0}(s,v(s,t)))\nonumber\\
&=& \mathrm{dist}(P_W(\td\Phi(s,t)),K^m_{\lambda,l_0}(s,0)\cup \partial|_{W}U^m_{\lambda,l_0}(s,0)).
\end{eqnarray}
Note that $\overline{\tilde{K}_\lambda\cup\tilde{U}^+_\lambda\cup\tilde{U}^-_\lambda}$ is the disjoint union of $\tilde{K}_\lambda\cup\partial\tilde{U}^+_\lambda\cup\partial\tilde{U}^-_\lambda$ and $\tilde{U}^+_\lambda\cup\tilde{U}^-_\lambda$. So if $(p,q)\not\in\tilde{U}^+_\lambda\cup\tilde{U}^-_\lambda$, then $P_W(\td\Phi(s,t))\in K^m_{\lambda,l_0}(s,0)\cup \partial|_{W}U^m_{\lambda,l_0}(s,0)$ by (\ref{5-14}), and so $d_i(s,t)=0$. This is a contradiction to (\ref{5-15}). Thus
\begin{equation}\label{5-27}
(p,q)\in\tilde{U}^+_\lambda\cup\tilde{U}^-_\lambda.
\end{equation}
Next, suppose that $\mathrm{dist}((p,s^mq),\partial\tilde{U}^m_\lambda(s))\leq\delta_i$. Since $\partial\tilde{U}^m_\lambda(s)$ is compact, we can choose a point $(\tilde{p},\tilde{q})\in \partial\tilde{U}^+_\lambda\cup\partial\tilde{U}^-_\lambda$ so that
$$
\left|(p,s^mq)-(\tilde{p},s^m\tilde{q})\right|=\mathrm{dist}((p,s^mq),\partial\tilde{U}^m_\lambda(s))\leq\delta_i.
$$
But $$
\left(
            \begin{array}{cc}
              \tilde{p} & l \\
              0 & s^m\tilde{q} \\
            \end{array}
          \right) \in \partial|_{W}U^m_{\lambda,l_0}(s,0),
$$
and so
\begin{eqnarray*}
\delta_i &<& \mathrm{dist}(P_W(\td\Phi(s,t)),K^m_{\lambda,l_0}(s,0)\cup \partial|_{W}U^m_{\lambda,l_0}(s,0)) \\
&\leq& \mathrm{dist}(P_W(\td\Phi(s,t)),\partial|_{W}U^m_{\lambda,l_0}(s,0)) \\
&\leq& \left|P_W(\td\Phi(s,t))-\left(
            \begin{array}{cc}
              \tilde{p} & l \\
              0 & s^m\tilde{q} \\
            \end{array}
          \right) \right|=\left|(p,s^mq)-(\tilde{p},s^m\tilde{q})\right|
\end{eqnarray*}
by (\ref{5-14}) and (\ref{5-15}). This is a contradiction, and we thus have
\begin{equation}\label{5-28}
\mathrm{dist}((p,s^mq),\partial\tilde{U}^m_\lambda(s))>\delta_i.
\end{equation}
Finally, suppose $|l|\geq l_0-\delta_i.$ Assume further that $l\geq l_0-\delta_i.$ Note
$$
\left(
            \begin{array}{cc}
              p & l_0 \\
              0 & s^mq \\
            \end{array}
          \right) \in \partial|_{W}U^m_{\lambda,l_0}(s,0),
$$
and so
\begin{eqnarray*}
\mathrm{dist}(P_W(\td\Phi(s,t)),K^m_{\lambda,l_0}(s,0)\cup \partial|_{W}U^m_{\lambda,l_0}(s,0)) &\leq& \mathrm{dist}(P_W(\td\Phi(s,t)),\partial|_{W}U^m_{\lambda,l_0}(s,0)) \\
&\leq& \left|P_W(\td\Phi(s,t))-\left(
            \begin{array}{cc}
              p & l_0 \\
              0 & s^m q \\
            \end{array}
          \right) \right|\\
&=& l_0-l\leq \delta_i
\end{eqnarray*}
by (\ref{5-14}). This is a contradiction to (\ref{5-15}), and thus
$
l<l_0-\delta_i.
$
If $l\leq -(l_0-\delta_i)$, then we also have a contradiction, so that we conclude that
\begin{equation}\label{5-29}
|l|<l_0-\delta_i.
\end{equation}
Thus (\ref{5-12}) follows from (\ref{5-14}), (\ref{5-27}), (\ref{5-28}), and (\ref{5-29}).

\subsection{Local gradient modifiers in subdivisions of $\hat{G}_i$}
By the Vitali Covering Lemma, we can take a sequence $\{Q^k_i\}_{k\in\NN}$ of disjoint open squares in $\hat{G}_i$ whose sides are parallel to the axes such that
$$
\left|\hat{G}_i\setminus\bigcup_{k=1}^\infty Q^k_i\right|=0.
$$
For each $k\in\NN$, let $d^k_i>0$ denote the side length of $Q^k_i$ and $(s^k_i,t^k_i)$ the center of $Q^k_i$. Dividing these squares further if necessary, we can have
\begin{equation}\label{5-10}
d^k_i\leq\min\left\{\frac{\eta_i}{\sqrt 2},\frac{4\eta}{\sqrt 2 (\lambda-\lambda^-)},\sqrt{\frac{4\eta}{\sqrt{2}(\lambda-\lambda^-)R^m}},\frac{\delta_i}{12M_g\sigma(1)}\right\}\,\,\,\forall k\in\NN,
\end{equation}
where $M_g:=\max_{|s_1|,|s_2|\leq R} |g(s_1,s_2)|$ and $g$ is the polynomial of two variables such that
$$
(s_1)^m-(s_2)^m=(s_1-s_2)g(s_1,s_2)\,\,\,\,\,\forall s_1,s_2\in\RR.\,\,\,\textrm{(Take $g\equiv 0$ if $m=0$.)}
$$
We fix an index $k\in\NN$ in the rest of the section. If $(s,t),(s',t')\in Q^k_i$, then $|(s,t)-(s',t')|<\sqrt{2}d^k_i\leq\eta_i$ by (\ref{5-10}), and so
$$
|\td\Phi(s,t)-\td\Phi(s',t')|\leq\rho\delta_i
$$
by (\ref{5-8}). In particular,
\begin{equation}\label{5-11}
|\td\Phi(s,t)-\td\Phi(s^k_i,t^k_i)|\leq\rho\delta_i\,\,\,\forall (s,t)\in Q^k_i.
\end{equation}

\begin{figure}[ht]
\begin{center}
\includegraphics[scale=0.6]{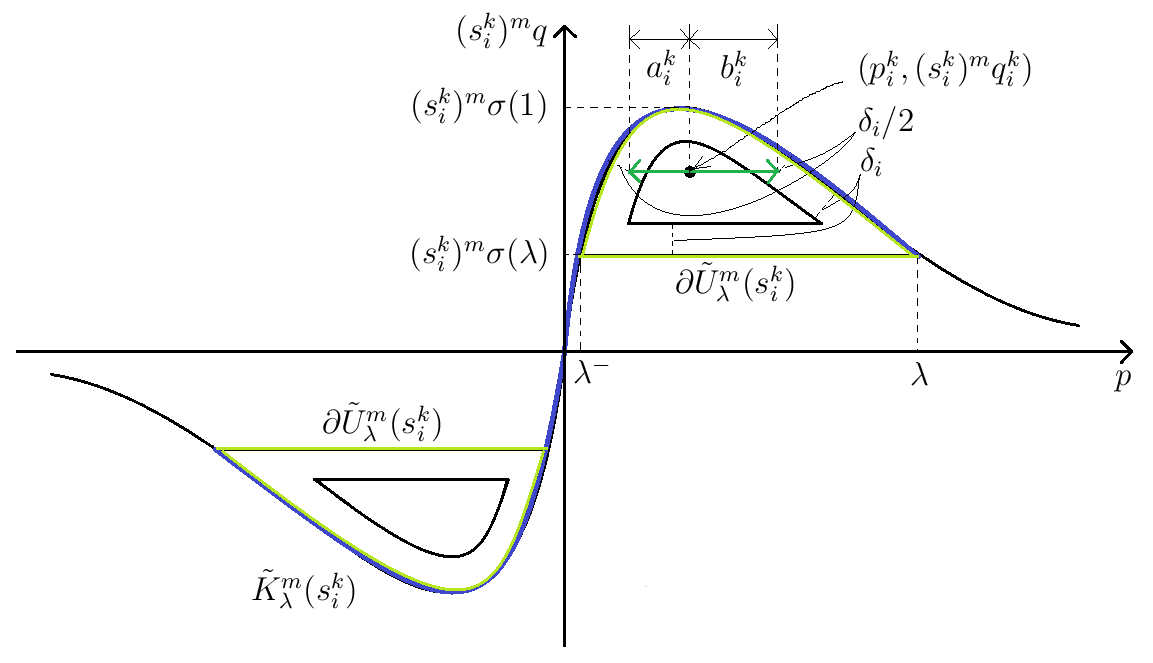}
\end{center}
\caption{A necessary local $s$-derivative change of $v$ in $Q^k_i$}
\label{fig4}
\end{figure}

Since $(s^k_i,t^k_i)\in Q^k_i\subset\hat{G}_i$, we have from (\ref{5-12}) that
\begin{equation}\label{5-34}
\left\{ \begin{array}{l}
          \td\Phi(s^k_i,t^k_i)\in U^m_{\lambda,l_0}(s^k_i,v(s^k_i,t^k_i)), \\
          \mathrm{dist}((\partial_s v(s^k_i,t^k_i),\partial_t\vp(s^k_i,t^k_i)),\partial\tilde{U}^m_\lambda(s^k_i))>\delta_i, \\
          |\partial_t v(s^k_i,t^k_i)|<l_0-\delta_i.
        \end{array} \right.
\end{equation}
So
\begin{equation}\label{5-36}
\left(
  \begin{array}{cc}
    \partial_s v(s^k_i,t^k_i) & \partial_t v(s^k_i,t^k_i) \\
    \partial_s \vp(s^k_i,t^k_i) & \partial_t \vp(s^k_i,t^k_i) \\
  \end{array}
\right) = \left(
            \begin{array}{cc}
              p^k_i & l^k_i \\
              (s^k_i)^m v(s^k_i,t^k_i) & (s^k_i)^mq^k_i \\
            \end{array}
          \right)
\end{equation}
for some $(p^k_i,q^k_i)\in \tilde{U}^+_\lambda\cup\tilde{U}^-_\lambda$ and some $l^k_i\in\RR$ with $|l^k_i|<l_0-\delta_i$. Also
\begin{equation}\label{5-37}
\mathrm{dist}((p^k_i,(s^k_i)^mq^k_i)),\partial\tilde{U}^m_\lambda(s^k_i))>\delta_i.
\end{equation}
So by the Intermediate Value Theorem, there exist two positive reals $a^k_i$ and $b^k_i$ such that
\begin{equation}\label{5-16}
\left\{ \begin{array}{l}
          \mathrm{dist}((p^k_i-a^k_i,(s^k_i)^mq^k_i),\tilde{K}^m_\lambda(s^k_i))=\mathrm{dist}((p^k_i+b^k_i,(s^k_i)^mq^k_i),\tilde{K}^m_\lambda(s^k_i))=\frac{\delta_i}{2}, \\
          (p^k_i-a^k_i,(s^k_i)^mq^k_i),(p^k_i+b^k_i,(s^k_i)^mq^k_i)\in \tilde{U}^m_\lambda(s^k_i).
        \end{array}
 \right.
\end{equation}
(See Figure \ref{fig4}.)
Observe $$
a^k_i+b^k_i<\lambda-\lambda^-.
$$
Let $\xi^k_i>0$ be a constant with
\begin{eqnarray}\label{5-17}
\xi^k_i\leq\min\left\{\frac{\delta_i}{2(\lambda-\lambda^-)\left[1+\left(\frac{R-\delta_0}{\delta_0}\right)^m\right]},\frac{\delta_i}{6(R-2\delta_0)(R-\delta_0)^m(\lambda-\lambda^-)\left[1+\left(\frac{R-\delta_0}{\delta_0}\right)^m\right]}            \right\}.
\end{eqnarray}
Define the diamond-shaped $\tilde{D}^k_i$ in $\RR^2$ as
$$
\tilde{D}^k_i:=\mathrm{int}\left(\mathrm{co}\{(0,1),(0,-1),(\xi^k_i,0),(-\xi^k_i,0)\}\right).
$$
By the Vitali Covering Lemma, there exist a sequence $\{(s^k_{i,j},t^k_{i,j})\}_{j\in\NN}$ in $Q^k_i$ and a sequence $\{\epsilon^k_{i,j}\}_{j\in\NN}$ of positive reals such that $\{(s^k_{i,j},t^k_{i,j})+\epsilon^k_{i,j} \tilde{D}^k_i\}_{j\in\NN}$ is a sequence of disjoint open subsets of $Q^k_i$ whose union has measure $|Q^k_i|$. Following the notations in (\ref{4-15}), we have
$$
(s^k_{i,j},t^k_{i,j})+\epsilon^k_{i,j} \tilde{D}^k_i=D(s^k_{i,j}-\epsilon^k_{i,j}\xi^k_{i},s^k_{i,j}+\epsilon^k_{i,j}\xi^k_{i},t^k_{i,j},\epsilon^k_{i,j})=:D^k_{i,j}\,\,\,\forall j\in\NN.
$$
Let $j\in\NN$. We also define according to the notations in (\ref{4-15}) that
\begin{equation*}
\begin{array}{rcl}
  (D^k_{i,j})^\pm_r & := & D^\pm_r(s^k_{i,j}-\epsilon^k_{i,j}\xi^k_{i},s^k_{i,j}+\epsilon^k_{i,j}\xi^k_{i},t^k_{i,j},\epsilon^k_{i,j})\,\,\,\,\,\forall r\in\{1,2,3\}, \\
  (s^k_{i,j})_r(t) & := & s_r(s^k_{i,j}-\epsilon^k_{i,j}\xi^k_{i},s^k_{i,j}+\epsilon^k_{i,j}\xi^k_{i},t^k_{i,j},\epsilon^k_{i,j};t)\,\,\,\,\,\forall t\in[t^k_{i,j}-\epsilon^k_{i,j},t^k_{i,j}+\epsilon^k_{i,j}],\,\forall r\in\{1,2\}, \\
  \tilde{v}^k_{i,j}(s,t) & := & \tilde{v}(-a^k_i,b^k_i,s^k_{i,j}-\epsilon^k_{i,j}\xi^k_{i},s^k_{i,j}+\epsilon^k_{i,j}\xi^k_{i},t^k_{i,j},\epsilon^k_{i,j};s,t)\,\,\,\,\,\forall (s,t)\in\overline{D^k_{i,j}}.
\end{array}
\end{equation*}
Then Lemma \ref{4-16} can be restated as follows in a bit more specific form:\\\\
\textbf{(a)} $\tilde{v}^k_{i,j}\in W^{1,\infty}_0(D^k_{i,j}),$\\
\textbf{(b)} $\tilde{v}^k_{i,j}\in C^1\left(\overline{(D^k_{i,j})^\pm_r}\right)\,\,\,\,\,\forall r\in\{1,2,3\},$\\
\textbf{(c)} $\partial_s\tilde{v}^k_{i,j}(s,t)=\left\{ \begin{array}{cl}
                                                         -a^k_i & \forall (s,t)\in (D^k_{i,j})^+_1\cup(D^k_{i,j})^-_1\cup(D^k_{i,j})^+_3\cup(D^k_{i,j})^-_3 \\
                                                         b^k_i & \forall (s,t)\in (D^k_{i,j})^+_2\cup(D^k_{i,j})^-_2,
                                                       \end{array}
  \right.   $\\
\textbf{(d)}
\begin{eqnarray*}
|\partial_t\tilde{v}^k_{i,j}(s,t)| &\leq& \max\{a^k_i,b^k_i\}\left[ 1+\left(\frac{s^k_{i,j}+\epsilon^k_{i,j}\xi^k_{i}}{s^k_{i,j}-\epsilon^k_{i,j}\xi^k_{i}}\right)^m\right]\frac{2\epsilon^k_{i,j}\xi^k_{i}}{2\epsilon^k_{i,j}}\\
&\leq& (\lambda-\lambda^-)\left[ 1+\left(\frac{R-\delta_0}{\delta_0}\right)^m\right]\xi^k_i\\
&\leq& \frac{\delta_i}{2}\,\,\,\,\,\forall (s,t)\in\bigcup_{r=1}^3[(D^k_{i,j})^+_r\cup (D^k_{i,j})^-_r],\,\,\,\,\, \textrm{(by (\ref{5-17}))}
\end{eqnarray*}
\textbf{(e)} $\frac{\partial}{\partial t}\left(\int^s_{(s^k_{i,j})_1(t)}\tau^m\tilde{v}^k_{i,j}(\tau,t)d\tau\right)=\int^s_{(s^k_{i,j})_1(t)}\tau^m\partial_t\tilde{v}^k_{i,j}(\tau,t)d\tau\,\,\,\,\,\forall (s,t)\in\bigcup_{r=1}^3[(D^k_{i,j})^+_r\cup (D^k_{i,j})^-_r],$\\
\textbf{(f)} $\int^{(s^k_{i,j})_2(t)}_{(s^k_{i,j})_1(t)}\tau^m\tilde{v}^k_{i,j}(\tau,t)d\tau=0\,\,\,\,\,\forall t\in[t^k_{i,j}-\epsilon^k_{i,j},t^k_{i,j}+\epsilon^k_{i,j}],$\\
\textbf{(g)} $\max_{\overline{D^k_{i,j}}}|\tilde{v}^k_{i,j}|\leq \frac{a^k+b^k_i}{4}2\epsilon^k_{i,j}\xi^k_{i}\leq \frac{\lambda-\lambda^-}{4}d^k_i\leq \frac{\eta}{\sqrt 2},\,\,\,\,\, \textrm{(by (\ref{5-10}))}$\\
\textbf{(h)}
\begin{eqnarray*}
\max_{(s,t)\in\overline{D^k_{i,j}}}\left|\int^s_{(s^k_{i,j})_1(t)}\tau^m\tilde{v}^k_{i,j}(\tau,t)d\tau\right| &\leq& \frac{a^k+b^k_i}{4}(s^k_{i,j}+\epsilon^k_{i,j}\xi^k_{i})^m(2\epsilon^k_{i,j}\xi^k_{i})^2\\
&\leq& \frac{\lambda-\lambda^-}{4}R^m(d^k_i)^2\leq \frac{\eta}{\sqrt 2}.\,\,\,\,\, \textrm{(by (\ref{5-10}))}
\end{eqnarray*}

\subsection{New function $\Phi_\eta$ from old $\Phi$}
We now define
\begin{equation*}
\tilde{v}:=\sum_{i,j,k\in\NN}\tilde{v}^k_{i,j}\chi_{D^k_{i,j}}\,\,\,\,\,\textrm{ in }J^*_T.
\end{equation*}
Note that $\forall i,j,k\in\NN$,
\begin{eqnarray*}
||\tilde{v}^k_{i,j}||_{W^{1,\infty}(D^k_{i,j})} &=& ||\tilde{v}^k_{i,j}||_{L^{\infty}(D^k_{i,j})}+||\partial_s\tilde{v}^k_{i,j}||_{L^{\infty}(D^k_{i,j})}+||\partial_t\tilde{v}^k_{i,j}||_{L^{\infty}(D^k_{i,j})}\\
&\leq& \frac{\eta}{\sqrt 2}+\max\{a^k_{i},b^k_j\}+\frac{\delta_i}{2}\,\,\,\,\,(\textrm{by (c), (d), and (g)})\\
&\leq& \frac{\eta}{\sqrt 2}+(\lambda-\lambda^-)+\frac{\epsilon}{4},\,\,\,\,\,(\textrm{by (\ref{5-5})})
\end{eqnarray*}
that is, $\sup_{i,j,k\in\NN}||\tilde{v}^k_{i,j}||_{W^{1,\infty}(D^k_{i,j})}\leq \frac{\eta}{\sqrt 2}+(\lambda-\lambda^-)+\frac{\epsilon}{4}<\infty$.
Applying the Gluing Lemma, it follows from this inequality, (a), and (b) that
\begin{equation}\label{5-19}
\left\{ \begin{array}{l}
          \tilde{v}\in W^{1,\infty}_0(J^*_T), \\
          \textrm{$\tilde{v}$ is piecewise $C^1$ in $J^*_T$.}
        \end{array}
 \right.
\end{equation}

Define
\begin{equation*}
\tilde{\vp}(s,t):=\int_{\delta_0}^s\tau^m\tilde{v}(\tau,t)d\tau\,\,\,\,\,\forall(s,t)\in \bar{J}^*_T.
\end{equation*}
It is then clear that $\tilde\vp\in W^{1,\infty}(J^*_T)$.
Also, by (f) and the definitions of $\tilde{\vp}$ and $\tilde{v}$,
\begin{equation}\label{5-21}
\tilde{\vp}(s,t)=0\,\,\,\,\,\forall (s,t)\in\bar{J}^*_T\setminus\bigcup_{i,j,k\in\NN}D^k_{i,j},
\end{equation}
and hence $\tilde\vp\equiv 0$ on $\partial J^*_T$. Thus $\tilde\vp\in W^{1,\infty}_0(J^*_T)$.

Let $r\in\{1,2,3\}$ and $(s,t)\in (D^k_{i,j})^\pm_r$. By (e) and (f),
\begin{equation}\label{5-23}
\partial_t\tilde{\vp}(s,t)=\int^s_{(s^k_{i,j})_1(t)}\tau^m\partial_t\tilde{v}^k_{i,j}(\tau,t)d\tau.
\end{equation}
So it is easily deduced from (b) that
$$
\partial_t\tilde{\vp}\in C^0\left(\overline{(D^k_{i,j})^\pm_r}\right).
$$
By the definition of $\tilde{\vp}$,
\begin{equation}\label{5-24}
\partial_s\tilde{\vp}(s,t)=s^m\tilde{v}(s,t)\,\,\,\,\,\forall (s,t)\in J^*_T.
\end{equation}
Since $\tilde{v}\in C^0(\bar{J}^*_T)$, we have $\partial_s\tilde{\vp}\in C^0(\bar{J}^*_T)$. In particular,
$$
\partial_s\tilde{\vp}\in C^0\left(\overline{(D^k_{i,j})^\pm_r}\right),
$$
so that
\begin{equation}\label{5-25}
\tilde{\vp}\in C^1\left(\overline{(D^k_{i,j})^\pm_r}\right).
\end{equation}
Thus
\begin{equation}\label{5-26}
\left\{ \begin{array}{l}
          \tilde{\vp}\in W^{1,\infty}_0(J^*_T), \\
          \textrm{$\tilde{\vp}$ is piecewise $C^1$ in $J^*_T$.}
        \end{array}
 \right.
\end{equation}
Finally, we define
\begin{equation*}
v_\eta:=v+\tilde{v},\,\,\vp_\eta:=\vp+\tilde{\vp},\,\,\textrm{and}\,\,\Phi_\eta:=(v_\eta,\vp_\eta)\,\,\,\,\textrm{ in $J^*_T$.}
\end{equation*}

\subsection{Completion of Proof of Theorem \ref{2-3-1}}
To finish the proof of the density theorem, Theorem \ref{2-3-1}, we will show that the function $\Phi_\eta$ defined above belongs to $\Phi^*+W^{1,\infty}_0(J^*_T;\RR^2)$ and satisfies all of (\ref{5-2}).

 First, it follows from (\ref{5-1}), (\ref{5-19}), and (\ref{5-26}) that
\begin{equation*}
\left\{ \begin{array}{l}
          \Phi_\eta\in\Phi^*+W^{1,\infty}_0(J^*_T;\RR^2), \\
          \textrm{$\Phi_\eta$ is piecewise $C^1$ in $J^*_T$.}
        \end{array}
 \right.
\end{equation*}
It remains to verify the rest of (\ref{5-2}).   \\

\underline{\textbf{The fourth of (\ref{5-2}):}} Note
$$
||\Phi-\Phi_\eta||_{L^\infty(J^*_T;\RR^2)}=\sup_{i,j,k\in\NN}||(\tilde{v},\tilde\vp)||_{L^\infty(D^k_{i,j};\RR^2)}\leq\sup_{i,j,k\in\NN}\left(||\tilde{v}||^2_{L^\infty(D^k_{i,j})}+||\tilde{\vp}||^2_{L^\infty(D^k_{i,j})}\right)^{1/2},
$$
since $\tilde{v}=\tilde{\vp}=0$ on $\bar{J}^*_T\setminus\bigcup_{i,j,k\in\NN}D^k_{i,j}$. But for every $(i,j,k)\in\NN^3$,
$$
||\tilde{v}||_{L^\infty(D^k_{i,j})}=||\tilde{v}^k_{i,j}||_{L^\infty(D^k_{i,j})}\leq\frac{\eta}{\sqrt2}\,\,\,\,\,(\textrm{by (g))}
$$
and
$$
||\tilde{\vp}||_{L^\infty(D^k_{i,j})}\leq\frac{\eta}{\sqrt2}.\,\,\,\,\,(\textrm{by (f) and (h))}
$$
Thus
$$
||\Phi-\Phi_\eta||_{L^\infty(J^*_T;\RR^2)}\leq\eta.
$$

\underline{\textbf{The second of (\ref{5-2}):}} By the third of (\ref{5-1}) and (\ref{5-24}),
\begin{eqnarray}\label{5-30}
\partial_s\vp_\eta(s,t) &=& \partial_s(\vp(s,t)+\tilde{\vp}(s,t))\nonumber\\
&=& s^m v(s,t)+s^m\tilde{v}(s,t)\\
&=& s^m v_\eta(s,t)\nonumber
\end{eqnarray}
for a.e. $(s,t)\in J^*_T$. Since $\Phi_\eta\equiv\Phi$ on $J^*_T\setminus\bigcup_{i,j,k\in\NN} D^k_{i,j}$, it follows from the third of (\ref{5-1}) that
$$
\td\Phi_\eta(s,t)=\td\Phi(s,t)\in K^m_{\lambda,l_0}(s,v(s,t))\cup U^m_{\lambda,l_0}(s,v(s,t))= K^m_{\lambda,l_0}(s,v_\eta(s,t))\cup U^m_{\lambda,l_0}(s,v_\eta(s,t))
$$
for a.e. $(s,t)\in J^*_T\setminus\bigcup_{i,j,k\in\NN} D^k_{i,j}$. Let $i,j,k\in\NN$. To finish the proof of this part, it now suffices to show that
\begin{equation}\label{5-31}
\td\Phi_\eta(s,t)\in U^m_{\lambda,l_0}(s,v_\eta(s,t))\,\,\,\,\,\textrm{ for a.e. }(s,t)\in D^k_{i,j}.
\end{equation}
To this end, we will show that for a.e. $(s,t)\in D^k_{i,j}$, we have
\begin{equation}\label{5-32}
|\partial_t v_\eta(s,t)|<l_0
\end{equation}
and
\begin{equation}\label{5-33}
\left\{ \begin{array}{l}
          \partial_s v_\eta(s,t)\in(-\lambda,-\lambda^-)\cup(\lambda^-,\lambda), \\
          \partial_t \vp_\eta(s,t)\in I^m_\lambda(s,\partial_s v_\eta(s,t)).
        \end{array}
 \right.
\end{equation}
Then combining (\ref{5-30}), (\ref{5-32}), and (\ref{5-33}) and appealing to Lemma \ref{2-3-9}, we obtain (\ref{5-31}).

Since $D^k_{i,j}\subset Q^k_i$, it follows from (\ref{5-11}) that for each $(s,t)\in D^k_{i,j}$,
$$
|\partial_t v(s,t)-\partial_t v(s^k_i,t^k_i)|\leq|\td\Phi(s,t)-\td\Phi(s^k_i,t^k_i)|\leq\rho\delta_i.
$$
But $|\partial_t v(s^k_i,t^k_i)|<l_0-\delta_i$ by the third of (\ref{5-34}). Observe also that for a.e. $(s,t)\in D^k_{i,j}$,
$$
|\partial_t\tilde{v}(s,t)|=|\partial_t\tilde{v}^k_{i,j}(s,t)|\leq\frac{\delta_i}{2}.\,\,\,\,\,\textrm{ (by (d))}
$$
Thus for a.e. $(s,t)\in D^k_{i,j}$,
\begin{eqnarray*}
|\partial_t v_\eta(s,t)| &=& |\partial_t v(s,t)+\partial_t \tilde{v}(s,t)|\\
&\leq& |\partial_t v(s,t)-\partial_t v(s^k_i,t^k_i)|+|\partial_t v(s^k_i,t^k_i)|+|\partial_t \tilde{v}(s,t)|\\
&<& \rho\delta_i+l_0-\delta_i+\frac{\delta_i}{2}<l_0-\frac{\delta_i}{3}<l_0,\,\,\,\,\,\textrm{ (by (\ref{5-9}))}
\end{eqnarray*}
and hence (\ref{5-32}) holds.

As above for each $(s,t)\in D^k_{i,j}$,
\begin{eqnarray}\label{5-35}
\rho\delta_i &\geq& |\td\Phi(s,t)-\td\Phi(s^k_i,t^k_i)|\nonumber\\
&\geq& |(\partial_s v(s,t),\partial_t \vp(s,t))-(\partial_s v(s^k_i,t^k_i),\partial_t \vp(s^k_i,t^k_i))|\nonumber\\
&=& |(\partial_s v(s,t),\partial_t \vp(s,t))-(p^k_i,(s^k_i)^mq^k_i)|\,\,\,\,\,\textrm{ (by (\ref{5-36}))}\nonumber\\
&\geq& \max\{|\partial_s v(s,t)-p^k_i|,|\partial_t \vp(s,t)-(s^k_i)^mq^k_i|\},
\end{eqnarray}
where $(p^k_i,q^k_i)\in\tilde{U}^+_\lambda\cup\tilde{U}^-_\lambda$. Let us assume that $(p^k_i,q^k_i)\in\tilde{U}^+_\lambda$. (The other case that $(p^k_i,q^k_i)\in\tilde{U}^-_\lambda$ can be shown in the same way.) We have to show that for a.e. $(s,t)\in D^k_{i,j}$,
\begin{equation}\label{5-38}
\left\{ \begin{array}{l}
          \lambda^-<\partial_s v_\eta(s,t)=\partial_s v(s,t)+\partial_s \tilde{v}(s,t)<\lambda, \\
          \partial_t \vp_\eta(s,t)\in I^m_\lambda(s,\partial_s v_\eta(s,t)), \\
          \textrm{or equivalently } s^m\sigma(\lambda)<\partial_t\vp(s,t)+\partial_t\tilde{\vp}(s,t)<s^m\sigma(\partial_s v(s,t)+\partial_s\tilde{v}(s,t)).
        \end{array}
 \right.
\end{equation}

\underline{\textbf{Case 1:}} Assume $(s,t)\in (D^k_{i,j})^+_1\cup(D^k_{i,j})^-_1\cup(D^k_{i,j})^+_3\cup(D^k_{i,j})^-_3$.\\
In this case, we have
$$
\partial_s \tilde{v}^k_{i,j}(s,t)=-a^k_i
$$
by (c). Let $0<p^{k,-}_i<1<p^{k,+}_i$ be such that
$$
\sigma(p^{k,\pm}_i)=q^k_i,
$$
so that $(s^k_i)^m\sigma(p^{k,\pm}_i)=(s^k_i)^mq^k_i$. Then by (\ref{5-16}),
$$
p^{k,-}_i+\frac{\delta_i}{3}<p^k_i-a^k_i<p^k_i<p^k_i+b^k_i<p^{k,+}_i-\frac{\delta_i}{3}.
$$
(See Figure \ref{fig4}.) Also by (\ref{5-9}) and (\ref{5-35}),
$$
-\frac{\delta_i}{6}<\partial_s v(s,t)-p^k_i<\frac{\delta_i}{6}.
$$
Thus
$$
\lambda^-<p^{k,-}_i<p^k_i-\frac{\delta_i}{3}-a^k_i<\partial_s v(s,t)+\frac{\delta_i}{6}-\frac{\delta_i}{3}-a^k_i
$$
$$
<\partial_s v(s,t)-a^k_i=\partial_s v(s,t)+\partial_s \tilde{v}^k_{i,j}(s,t)=\partial_s v_\eta(s,t)
$$
$$
<\partial_s v(s,t)-\frac{\delta_i}{6}+b^k_i+\frac{\delta_i}{3}<p^k_i+b^k_i+\frac{\delta_i}{3}<p^{k,+}_i<\lambda,
$$
that is,
\begin{equation}\label{5-39}
\lambda^-<\partial_s v_\eta(s,t)<\lambda.
\end{equation}
Next, note from (\ref{5-17}), (\ref{5-23}), and (d) that
\begin{equation}\label{5-40}
|\partial_t\tilde{\vp}(s,t)|\leq (R-2\delta_0)(R-\delta_0)^m(\lambda-\lambda^-)\left[1+\left(\frac{R-\delta_0}{\delta_0}\right)^m\right]\xi^k_i\leq\frac{\delta_i}{6}.
\end{equation}
By (\ref{5-16}),
\begin{equation}\label{5-41}
(s^k_i)^mq^k_i\leq (s^k_i)^m\sigma(p^k_i-a^k_i)-\frac{\delta_i}{2}.
\end{equation}
(See Figure \ref{fig4}.)
But
\begin{equation*}
\begin{array}{c}
  |(s^k_i)^m\sigma(p^k_i-a^k_i)-s^m\sigma(\partial_s v(s,t)-a^k_i)| \\
  =|((s^k_i)^m-s^m)\sigma(p^k_i-a^k_i)+s^m(\sigma(p^k_i-a^k_i)-\sigma(\partial_s v(s,t)-a^k_i))| \\
  \leq |s^k_i-s||g(s^k_i,s)|\sigma(1)+R^mM_\sigma|p^k_i-\partial_s v(s,t)| \\
  \leq M_g\sigma(1)d^k_i+R^mM_\sigma\rho\delta_i\,\,\,\,\,\,\textrm{(by (\ref{5-35}))} \\
  <\frac{\delta_i}{6}.\,\,\,\,\,\,\textrm{(by (\ref{5-9}) and (\ref{5-10}))}
\end{array}
\end{equation*}
Combining this with (\ref{5-41}), we get
\begin{equation}\label{5-42}
(s^k_i)^mq^k_i<s^m\sigma(\partial_s v(s,t)-a^k_i)-\frac{\delta_i}{3}.
\end{equation}
So
\begin{eqnarray}\label{5-43}
\partial_t\vp(s,t)+\partial_t\tilde{\vp}(s,t) &\leq& (s^k_i)^mq^k_i+\rho\delta_i+|\partial_t\tilde{\vp}(s,t)|\,\,\,\,\,\,\textrm{(by (\ref{5-35}))}\nonumber\\
&<& s^m\sigma(\partial_s v(s,t)-a^k_i)-\frac{\delta_i}{3}+\frac{\delta_i}{6}+\frac{\delta_i}{6}\,\,\,\,\,\,\textrm{(by (\ref{5-9}), (\ref{5-40}), and (\ref{5-42}))}\nonumber\\
&=& s^m\sigma(\partial_s v(s,t)+\partial_s \tilde{v}(s,t)).
\end{eqnarray}
Note (See Figure \ref{fig4}.)
\begin{equation}\label{5-44}
(s^k_i)^mq^k_i>(s^k_i)^m\sigma(\lambda)+\delta_i\,\,\,\,\,\,\textrm{(by (\ref{5-37}))}
\end{equation}
and
\begin{equation}\label{5-45}
|s^m\sigma(\lambda)-(s^k_i)^m\sigma(\lambda)|\leq M_g\sigma(1)d^k_i\leq\frac{\delta_i}{12}.\,\,\,\,\,\,\textrm{(by (\ref{5-10}))}
\end{equation}
So
\begin{eqnarray}\label{5-46}
\partial_t\vp(s,t)+\partial_t\tilde{\vp}(s,t) &\geq& (s^k_i)^mq^k_i-\rho\delta_i-|\partial_t\tilde{\vp}(s,t)|\,\,\,\,\,\,\textrm{(by (\ref{5-35}))}\nonumber\\
&\geq& (s^k_i)^m\sigma(\lambda)+\delta_i-\frac{\delta_i}{6}-\frac{\delta_i}{6}\,\,\,\,\,\,\textrm{(by (\ref{5-9}), (\ref{5-40}), and (\ref{5-44}))}\nonumber\\
&\geq& s^m\sigma(\lambda)-\frac{\delta_i}{12}+\delta_i-\frac{\delta_i}{6}-\frac{\delta_i}{6}\,\,\,\,\,\,\textrm{(by (\ref{5-45}))}\nonumber\\
&>& s^m\sigma(\lambda).
\end{eqnarray}
Combining (\ref{5-39}), (\ref{5-43}), and (\ref{5-46}), we have (\ref{5-38}) whenever $(s,t)\in (D^k_{i,j})^+_1\cup(D^k_{i,j})^-_1\cup(D^k_{i,j})^+_3\cup(D^k_{i,j})^-_3$.

\underline{\textbf{Case 2:}} (\ref{5-38}) also holds whenever $(s,t)\in (D^k_{i,j})^+_2\cup(D^k_{i,j})^-_2$. To show this, we just follow the lines of Case 1 with minor modifications whenever it is necessary. We skip the details.\\
We conclude from Cases 1 and 2 that (\ref{5-38}) holds for a.e. $(s,t)\in D^k_{i,j}$.

\underline{\textbf{The third of (\ref{5-2}):}} Observe
$$
\int_{J^*_T}\mathrm{dist}(\td\Phi_\eta(s,t), K^m_{\lambda,l_0}(s,v_\eta(s,t)))dsdt = \sum_{i=1}^\infty\int_{\hat{K}_i}\mathrm{dist}(\td\Phi(s,t), K^m_{\lambda,l_0}(s,v(s,t)))dsdt
$$
$$
+\sum_{i,j,k\in\NN}^\infty\int_{D^k_{i,j}}\mathrm{dist}(\td\Phi(s,t)+\td\tilde{\Phi}(s,t), K^m_{\lambda,l_0}(s,v(s,t)+\tilde{v}(s,t)))dsdt=:A+B,
$$
where $\tilde{\Phi}:=(\tilde{v},\tilde{\vp}).$
By (\ref{5-7}), we have $A\leq\frac{\epsilon}{2}|J^*_T|$. Let $i,j,k\in\NN$, and let $(s,t)\in\bigcup_{r=1}^3[(D^k_{i
,j})^+_r\cup (D^k_{i,j})^-_r]$ be any point at which (\ref{5-31}) holds. Then by Lemma \ref{2-3-6},
$$
\mathrm{dist}(\td\Phi(s,t)+\td\tilde{\Phi}(s,t),K^m_{\lambda,l_0}(s,v(s,t)+\tilde{v}(s,t)))
$$
$$
=\mathrm{dist}((\partial_s v(s,t)+\partial_s\tilde{v}(s,t),\partial_t \vp(s,t)+\partial_t\tilde{\vp}(s,t))    ,\tilde{K}^m_{\lambda}(s)).
$$
We assume further that $(s,t)\in(D^k_{i,j})^+_1\cup(D^k_{i,j})^-_1\cup(D^k_{i,j})^+_3\cup(D^k_{i,j})^+_3$, so that $\partial_s\tilde{v}(s,t)=-a^k_i$. Choose any $(p,q)\in\tilde{K}_\lambda$. Then
$$
|(\partial_s v(s,t)-a^k_i,\partial_t \vp(s,t)+\partial_t\tilde{\vp}(s,t))-(p,s^mq)|
$$
$$
=|(\partial_s v(s,t),\partial_t \vp(s,t))-(\partial_s v(s^k_i,t^k_i),\partial_t \vp(s^k_i,t^k_i))+(\partial_s v(s^k_i,t^k_i),\partial_t \vp(s^k_i,t^k_i))
$$
$$
+(-a^k_i,0)+(0,\partial_t\tilde{\vp}(s,t))-(p,s^mq)-(p,(s^k_i)^mq)+(p,(s^k_i)^mq)|
$$
$$
\leq\rho\delta_i+|(p^k_i-a^k_i,(s^k_i)^mq^k_i)-(p,(s^k_i)^mq)|
$$
$$
+|\partial_t\tilde{\vp}(s,t)|+|q||s^m-(s^k_i)^m|
$$
$$
\leq\frac{\delta_i}{6}+\frac{\delta_i}{6}+\frac{\delta_i}{12}+|(p^k_i-a^k_i,(s^k_i)^mq^k_i)-(p,(s^k_i)^mq)|
$$
as in the verification for the second of (\ref{5-2}). Taking an infimum on $(p,q)\in\tilde{K}_\lambda$ for the far-left and -right terms of the inequalities, we have
$$
\mathrm{dist}((\partial_s v(s,t)+\partial_s\tilde{v}(s,t),\partial_t \vp(s,t)+\partial_t\tilde{\vp}(s,t))    ,\tilde{K}^m_{\lambda}(s))
$$
$$
\leq \frac{5\delta_i}{12}+\mathrm{dist}((p^k_i-a^k_i,(s^k_i)^mq^k_i),\tilde{K}^m_\lambda(s^k_i))=\frac{5\delta_i}{12}+ \frac{\delta_i}{2}<\delta_i<\frac{\epsilon}{2}
$$
by (\ref{5-5}) and (\ref{5-16}). We  can get the same result when $(s,t)\in(D^k_{i,j})^+_2\cup(D^k_{i,j})^-_2$, but we omit the details. We now have
$$
\mathrm{dist}((\partial_s v(s,t)+\partial_s\tilde{v}(s,t),\partial_t \vp(s,t)+\partial_t\tilde{\vp}(s,t))    ,\tilde{K}^m_{\lambda}(s))\leq\frac{\epsilon}{2}\,\,\,\,\,\,\textrm{ for a.e. }(s,t)\in D^k_{i,j}.
$$
So we obtain $B\leq\sum_{i,j,k\in\NN}\frac{\epsilon}{2}|D^k_{i,j}|\leq\frac{\epsilon}{2}|J^*_T|.$ Thus $A+B\leq\epsilon|J^*_T|$.

The theorem is finally proved.


\addcontentsline{toc}{section}{References}

\bibliographystyle{plain}

\bibliography{paper}

\end{document}